\documentclass[12pt,a4paper]{article}
\usepackage{amsthm,amsmath,amssymb,cite,hyperref,color, bookmark,graphicx}

\newtheorem{proposition}{Proposition}
\newtheorem{corollary}{Corollary}
\newtheorem{lemma}{Lemma}
\newtheorem{theorem}{Theorem}

\newtheorem{remark}[theorem]{Remark}

\theoremstyle{definition}

\newtheorem{example}{Example}

\newcommand{\Ball}{\mathcal{B}}
\newcommand{\D}{\mathcal{D}}

\newcommand{\cl}{\mathrm{cl}}
\newcommand{\co}{\mathrm{conv}}
\newcommand{\N}{\mathbb{N}}
\newcommand{\oh}{o}
\newcommand{\R}{\mathbb{R}}
\newcommand{\Sph}{\mathcal{S}}

\DeclareMathOperator*{\Argmin}{Arg\, min}
\DeclareMathOperator*{\Argmax}{Arg\, max}
\DeclareMathOperator*{\dist}{dist}
\DeclareMathOperator*{\Limsup}{Lim\, sup}
\DeclareMathOperator*{\Er}{Er}

\title{Outer limits of subdifferentials for min-max type functions}
\author{Andrew Eberhard\thanks{School of Science, RMIT University}, Vera Roshchina$^*$\thanks{CIAO, Federation University Australia}, Tian Sang$^*$}

\begin{document}

\maketitle

\begin{abstract}
We generalise the outer subdifferential construction suggested by C\'anovas, Henrion, L\'opez and Parra for max type functions to pointwise minima of regular Lipschitz functions. We also answer an open question about the relation between the outer subdifferential of the support of a regular function and the end set of its subdifferential posed by Li, Meng and Yang.
\end{abstract}

\section{Introduction}

Our motivation for the study of outer limits of subdifferentials is the problem of constructive evaluation of error bounds. The error bound modulus measures whether a given function is steep enough locally outside of its level set. This idea stems from the works of Hoffman \cite{Hoffman:1952} and {\L}ojasiewicz \cite{Lojasiewicz}. Error bounds are crucial for a range of stability questions, for the existence of exact penalty functions, and for the convergence of numerical methods. The literature on error bounds is vast, and we refer the reader to the following selection of recent works and classic review papers for more details \cite{AzeCorvellec,Kruger:2014:1,Kruger:2015:2,Herion:2001,Fabian:2010,NgaiThera,WuYe,AzeSurvey,Aze:2002}. % An error bound can be viewed as a restricted version of calmness when the behaviour of the function is only considered outside of the solution set.
In this work we focus on the constructive evaluation of error bound modulus for structured continuous functions.

Let $f:X\to \R$ be a continuous function defined on an open set $X\subseteq \R^n$. We define the sublevel set
$$
S(\bar x) = \{x\in X\, |\, f(x)\leq f(\bar x)\},
$$
where $\bar x\in X$. We say that $f$ has a local (linear) error bound at $\bar x$ if there exists a constant $L>0$ such that
\begin{equation}\label{eq:error-bound}
L\dist\left(  x,S(\bar x)\right)  \leq \max\{0, f\left(  x\right)  -f\left(  \bar{x}\right)\}
\end{equation}
for all points $x$ in a sufficiently small neighbourhood of $\bar x$. Here $\dist (x,A)= \inf_{v\in A}\|x-a\|$ is the distance from $x$ to $A$. Taking the supremum over all constants $L$ that satisfy \eqref{eq:error-bound} over all neighbourhoods of  $\bar x$ we arrive at an exact quantity called the error bound modulus of $f$ at $\bar x$, which can be explicitly expressed as
\[
\Er f\left(  \bar{x}\right)  :=\liminf_{\substack{x\rightarrow
		\bar{x}\\f\left(  x\right)  >f\left(  \bar{x}\right)  }}\frac{f\left(
	x\right)  -f\left(  \bar{x}\right)  }{\dist\left(  x,S(\bar x)  \right)  }.
\]
It is possible to obtain sharp estimates of the error bound modulus $\Er f\left(  \bar{x}\right) $ for sufficiently structured functions by means of subdifferential calculus. For continuous functions that we are considering in this paper the error bound modulus is bounded from below by the distance from zero to the outer limits of Fr\'echet subdifferentials (see \cite{Fabian:2010}),
\begin{equation}\label{eq:main-errb-outer}
\Er f\left(  \bar{x}\right)  \geq \dist \left( 0, \Limsup_{x\to \bar x \atop f\left(x\right)  \downarrow f\left(\bar{x}\right)}\partial f(x)\right),
\end{equation}
with equality holding when $f$ is sufficiently  regular (for instance convex), see \cite[Theorem 5 and Proposition 10]{Fabian:2010}. In \cite{KaiwenMinghua} this equality is proved for a lower $C^1$ function and an additional upper estimate of the error bound of a regular locally Lipschitz function is given via the distance to the outer limits of the Fr\'echet subdifferentials of the subdifferential support function, and such limits are in turn expressed using the notion of the {\em end} of a closed convex set introduced in \cite{Hu}. We emphasise here that the inequality \eqref{eq:main-errb-outer} can be used to establish the inequality $\Er f\left(  \bar{x}\right)>0$, hence, the computation of the outer limits of subdifferentials in the right-hand side of \eqref{eq:main-errb-outer} is of significant interest even when the equality does not hold. Note that the outer limits of subdifferentials are called outer (limiting) subdifferentials (see \cite{AzeSurvey}).

In this paper we generalise some of the constructive results of \cite{OuterLimits} to the case of min-max type functions, providing an exact description for the outer limits of subdifferentials in the case of polyhedral functions and sharp bounds for a more general case (see Theorems~\ref{con:gen3.2} and \ref{thm:gen3.2lin}). We also strengthen Theorem~3.2 of \cite{OuterLimits} in Corollary~\ref{cor:gen3.2lin} by dropping the affine independence assumption (although the latter result can probably be obtained from the findings of \cite{KaiwenMinghua}). Finally, we answer in the affirmative the open question of \cite{KaiwenMinghua} for the case of functions with sublinear Hadamard directional derivative (see Corollary~\ref{cor:gen3.2} and Remark~\ref{rem:KLY}).

As we work in a finite-dimensional real space, throughout the paper we use the standard scalar product $\langle x, y\rangle = x^T y$, and denote the Euclidean norm by $\|x\|$. We also denote the closed unit ball and the unit sphere by $\Ball$ and $\Sph $ respectively.

\section{Preliminaries}

Recall that a function $f:X\to \R$, where $X$ is an open subset of $\R^n$, is {\em Hadamard directionally differentiable} at $x\in X$ for $p\in \R^n$ if the limit
$$
f'(x;p) = \lim_{t\downarrow 0 \atop p'\to p} \frac{f(x+tp')- f(x)}{t}
$$
exists and is finite. The quantity $f'(x;p)$ is called the (Hadamard) {\em directional derivative} of $f$ at $x$ in the direction $p$. It follows from the definition that the directional derivative is a positively homogeneous function of degree one, i.e. if $f$ is Hadamard directionally differentiable at $x\in X$, then
\begin{equation}\label{eq:HadDDHomogeneous}
f'(x;\lambda p) = \lambda f'(x;p) \quad \forall \, p\in \R^n, \, \lambda>0.
\end{equation}
Hadamard directionally differentiable functions enjoy certain continuity properties that we summarise in the next proposition. These properties are well-known (e.g. see \cite{DemRub}), but we provide a proof here for convenience. Throughout the paper we assume that $X$ is an open subset of $\R^n$.
\begin{proposition} Let $f:X\to \R$ be Hadamard directionally differentiable at $\bar x\in X$. Then the directional derivative $f'(\bar x;\cdot)$ is a continuous function; moreover,
\begin{equation}\label{eq:HadProp}
	f(\bar x+s) = f(\bar x) +f'(\bar x;s) + \oh (s), \qquad \frac{\oh(s)}{\|s\|} \underset{s \to 0}{\longrightarrow} 0.
\end{equation}
\end{proposition}
\begin{proof} We first show that the Hadamard directional derivative is continuous. Choose an arbitrary $p\in \R^n$ and a sequence $\{p_k\}$, $p_k \to p$.
By the definition of the directional derivative for every $k\in \N$ there exist $t_k$ such that $0<t_k<1/k$ and
\begin{equation}\label{eq:pf01}
f'(\bar x;p_k) = \frac{f(\bar x + t_k p_k)- f(\bar x)}{t_k}+ \delta_k, \quad |\delta_k|<1/k.
\end{equation}
Since $t_k\downarrow 0$, $p_k\to p$, and $f$ is Hadamard directionally differentiable at $\bar x$, we have
\begin{equation}\label{eq:pf02}
\lim_{k\to \infty} \frac{f(\bar x+t_kp_k)-f(\bar x)}{t_k} = f'(\bar x;p).
\end{equation}
Now passing to the limit on both sides of \eqref{eq:pf01} and using \eqref{eq:pf02}, we obtain
$$
\lim_{k\to \infty} f'(\bar x;p_k)  = \lim_{k\to \infty} \frac{f(\bar x+t_kp_k)-f(\bar x)}{t_k} + \lim_{k\to \infty} \delta_k = f'(\bar x;p),
$$
and so the directional derivative is continuous.

It remains to show the relation \eqref{eq:HadProp}. Assume the contrary. Then there is $\bar x\in X$, a constant $c>0$ and a sequence $\{s_k\}$, $s_k \to 0$ such that
\begin{equation*}%\label{eq:pf03}
\frac{|f(\bar x+s_k) - f(\bar x) -f'(\bar x;s_k)|}{\|s_k\|} > c \quad \forall k\in \N.
\end{equation*}
Without loss of generality we can assume that $s_k/\|s_k\| =: p_k\to p\in \Sph$, then from the continuity of $f'(\bar x;\cdot)$ we get
\begin{align*}
0 & =  \left| f'(\bar x;p)- \lim_{k\to \infty}  f'(\bar x;p_k)\right|\\	
& =  \left| f'(\bar x;p)- \lim_{k\to \infty}  \frac{f'(\bar x;s_k)}{\|s_k\|}\right|\qquad \text{(using \eqref{eq:HadDDHomogeneous})}\\
& =  \left| \lim_{k\to \infty} \frac{f(\bar x+\|s_k\|p_k) - f(\bar x)}{\|s_k\|} - \lim_{k\to \infty}  \frac{f'(\bar x;s_k)}{\|s_k\|} \right|\\ %& \text{(since $s_k = \|s_k\|\cdot p_k$)}\\
& = \lim_{k\to \infty} \frac{|f(\bar x+s_k) - f(\bar x) -f'(\bar x;s_k)|}{\|s_k\|}\geq c,% & \text{(taking limits on both sides of \eqref{eq:pf03})}\\
\end{align*}
which is impossible by our assumption that $c>0$.
\end{proof}

In this work our focus is on the functions with sublinear Hadamard directional derivatives and finite minima of such functions. The former are called subdifferentiable functions in \cite{DemRub}, however this notation is not universally accepted (e.g. in \cite{Kruger} subdifferentiable functions are the ones with nonempty Fr\'echet subdifferential). To avoid possible confusion with definitions, throughout the paper we sacrifice brevity for clarity and use the full description. Regular Lipschitz functions have sublinear Hadamard directional derivatives, see \cite[Theorem~9.16]{RockWets:1998}, therefore all results obtained here for functions with sublinear Hadamard directional derivatives also apply to regular Lipschitz functions.

%There is little need to distinguishing between different simple subdifferentials of functions with sublinear Hadamard directional derivatives, as most classic constructions coincide for this class (see \cite{Kruger}).

Recall that the Fr\'{e}chet subdifferential of a function $f:X\rightarrow\mathbb{R}$ at $\bar{x}\in X$ is the set
\[
\partial f(\bar{x})=\left\{  v\in\mathbb{R}^{n}\,\left\vert \,\liminf
_{x\rightarrow\bar{x},x\neq\bar{x}}\frac{f(x)-f(\bar{x})-\langle v,x-\bar
	{x}\rangle}{\Vert x-\bar{x}\Vert}\geq0\right.  \right\}.
\]
For Hadamard directionally differentiable function $f:X\to \R$ one has (see \cite[Proposition 1.17]{Kruger})
\begin{equation}\label{eq:FrechHad}
\partial f(\bar x) = \left\{ v\in \R^n\,\left|\, f'(\bar x;p)\geq \langle v,p\rangle \; \forall p\in \R^n \right. \right\}.
\end{equation}
When the directional derivative is sublinear, the Fr\'echet subdifferential of $f$ at $\bar x \in X$ coincides with the subdifferential of the directional derivative at $0$, so we have
\begin{equation}\label{eq:FrechHadSuppGeneral}
f'(\bar x;p) = \max_{v\in \partial f(\bar x )}\langle v,p\rangle \quad \forall p\in \R^n;
\end{equation}
moreover, for a convex function $f:\R^n\to \R$ the Fr\'echet subdifferential coincides with the classic Moreau-Rockafellar subdifferential,
\begin{align}\label{eq:FrechConv}
\partial f (x)
& = \{v\in \R^n \, |\,  f(y)-f(x) \geq \langle v,y-x\rangle  \; \forall\, y\in \R^n \}\notag\\
& = \{v\in \R^n \, |\,  f'(x;p)\geq \langle v,p\rangle  \;\forall\, p \in \Sph \}.
\end{align}
We will be using the following result explicitly (see \cite[ Chap. VI, Example 3.1]{ConvAnalVolOne}).
\begin{proposition}\label{prop:classic-subdiff} Let $h:\R^n\to \R$ be a sublinear function,
$$
h(x) = \max_{v\in C} \langle x, v\rangle,
$$
where $C \subset \R^n$ is a compact convex set. Then
\begin{equation}\label{eq:subdiff-classic-boundary}
\partial h(x) = \Argmax_{v\in C} \langle x, v\rangle.
\end{equation}
\end{proposition}
%\begin{proof} First of all, observe that
%$$
%\Argmax_{v\in C} \langle x, v\rangle \subset \partial h(x).
%$$
%indeed, for any $v\in \partial h(\bar x)$ and any $x\in \R^n$ we have by the homonegeity of $h$
%$$
%h(x)-h(t \bar x) \geq \langle v, x-t\bar x\rangle \forall x\in \R^n, \forall t>0.
%$$
%Letting $t\to 0$, we have $h(t\bar x ) = t h(\bar x)\to 0$, $\langle x, t\bar x\rangle \to 0$, hence,
%$$
%h(x)\geq \langle v, x\rangle \quad \forall x\in \R^n,
%$$
%and so $v\in C$. On the other hand,
%$$
%h(1/2\bar x)-h(\bar x) =
%$$
%\end{proof}

We will also utilise the following optimality condition (see Corollary~1.12.3 in \cite{Kruger}).

\begin{proposition}\label{prop:optimality} Let $ f_1 : X \to \R$ and $f_2 : X \to \R$  and assume that $f_1$ is Fr\'echet differentiable at $x$. If $f_1+f_2$ attains a local minimum at $x$, then $-\nabla f_1(x) \subset \partial f_2(x)$.
\end{proposition}

Let $f:X\to \R$ be a pointwise minimum of a finite set of functions with sublinear Hadamard directional derivatives. We have explicitly
\begin{equation}\label{eq:min-function}
f (x) = \min_{i\in I} f_i(x) \quad \forall x\in X,
\end{equation}
where $f_i:X\to \R$ are Hadamard directionally differentiable at $\bar x\in X$ with sublinear directional derivatives, so that (see \eqref{eq:FrechHadSuppGeneral})
\begin{equation}\label{eq:FrechHadSupp}
	f'_i(\bar x;p) = \max_{v\in \partial f_i(\bar x )}\langle v,p\rangle \quad \forall p\in \R^n,
\end{equation}
and $I$ is a finite index set. Observe that the relation \eqref{eq:HadProp} is valid for each individual function $f_i$, $i\in I$, so that we have
\begin{equation}\label{eq:HadPropIndividual}
	f_i(\bar x+s) = f_i(\bar x) +f_i'(\bar x;s) + \oh_i (s), \qquad \frac{\oh_i(s)}{\|s\|} \underset{s \to 0}{\longrightarrow} 0.
\end{equation}
Throughout the paper we use the following two active index sets.
$$
I(x) = \{i\in I\,|\, f(x) = f_i(x) \},
$$
\begin{equation}\label{eq:Ip}
I(x,p) = \left\{i_0\in I(x) \, \Bigl|\, \max_{v\in \partial f_{i_0}(x)}\langle v,p\rangle = \min_{i\in I(x)}\max_{v\in \partial f_i(x)}\langle v,p\rangle\right\}.
\end{equation}

We will need the following  well known relation (see \cite{DemRub}).
\begin{proposition}\label{prop:ddmin} Let $f:X\to \R$ be a pointwise minimum of a finite number of functions with sublinear Hadamard directional derivatives, as in \eqref{eq:min-function}. Then
$$
	f'(\bar x ; p) = \min_{i\in I(\bar x)} f_i'(\bar x;p) \quad \forall\, p \in \Sph.
$$
\end{proposition}
\begin{proof} The proof follows from the definition of directional derivative. Indeed, for all $x$ in a sufficiently small neighbourhood of $\bar x$ we have $I(x) \subset I(\bar x)$. Therefore
\begin{align*}
	\lim_{t\downarrow 0\atop p'\to p}\frac{f(\bar x + t p')- f(\bar x)}{t}
	& =  \lim_{t\downarrow 0\atop p'\to p}\frac{\min_{i\in I(\bar x)}f_i(\bar x + t p')- f(\bar x)}{t}\\
	& =  \lim_{t\downarrow 0\atop p'\to p}\frac{\min_{i\in I(\bar x)}[f_i(\bar x + t p')- f_i(\bar x)]}{t}\\
	& =  \min_{i\in I(\bar x)}\lim_{t\downarrow 0\atop p'\to p}\frac{f_i(\bar x + t p')- f_i(\bar x)}{t} = \min_{i\in I(\bar x)} f_i'(\bar x,p).
\end{align*}
\end{proof}

The next relation is well known (see \cite{FrechetCalculus} for the discussion of more general calculus rules for Fr\'echet subdifferentials) and follows directly from the definition of the Fr\'echet subdifferential and Proposition~\ref{prop:ddmin}. We provide a proof here for the sake of completeness.

\begin{proposition}\label{prop:calc-intersection} Let $f:X\to \R$ be a finite minimum of Hadamard directionally differentiable functions with sublinear derivatives at $\bar x \in X$. Then the Fr\'echet subdifferential of $f$ at $\bar x \in X$ is the intersection of the Fr\'echet subdifferentials of the active functions. In other words, given $f:X\to \R$ such that
	$$
	f(x) = \min_{i\in I} f_i(x) \quad \forall x\in X,
	$$
	where $I$ is a finite index set, and $f_i:X\to \R $ are Hadamard directionally differentiable with sublinear directional derivatives at $\bar x\in X$, one has
	$$
	\partial f(\bar x) = \bigcap_{i\in I(\bar x)} \partial f_i(\bar x).
	$$
%	where
%	$$
%	I(\bar x) = \{i \in I \, |\, f(\bar x) = f_i(\bar x) \}.
%	$$
\end{proposition}
\begin{proof} First of all, from Proposition~\ref{prop:ddmin} we have
\begin{equation}\label{eq:pf06}
f'(\bar x ; p) = \min_{i\in I(\bar x)} f_i'(\bar x;p) \quad \forall\, p \in \Sph.
\end{equation}
%which follows from the definition of directional derivative. Indeed, for all $x$ in a sufficiently small neighbourhood of $\bar x$ we have $I(x) \subset I(\bar x)$. Therefore
%\begin{align*}
%\lim_{t\downarrow 0\atop p'\to p}\frac{f(\bar x + t p')- f(\bar x)}{t}
% & =  \lim_{t\downarrow 0\atop p'\to p}\frac{\min_{i\in I(\bar x)}f_i(\bar x + t p')- f(\bar x)}{t}\\
% & =  \lim_{t\downarrow 0\atop p'\to p}\frac{\min_{i\in I(\bar x)}[f_i(\bar x + t p')- f(\bar x)]}{t}\\
% & =  \min_{i\in I(\bar x)}\lim_{t\downarrow 0\atop p'\to p}\frac{f_i(\bar x + t p')- f_i(\bar x)}{t} = \min_{i\in I(\bar x)} f_i'(\bar x,p).
%\end{align*}

Using \eqref{eq:FrechHad} and \eqref{eq:pf06} we have
\begin{align*}
\partial f(\bar x)
& = \left\{ v\in \R^n\,\left|\, \min_{i\in I(\bar x)}f_i'(\bar x;p)\geq \langle v,p\rangle \; \forall p\in \R^n \right. \right\}\\
& = \left\{ v\in \R^n\,\left|\, f_i'(\bar x;p)\geq \langle v,p\rangle \; \forall p\in \R^n \quad \forall i \in I(\bar x)\right. \right\}\\
& = \bigcap_{i\in I(\bar x)}\left\{ v\in \R^n\,\left|\, f_i'(\bar x;p)\geq \langle v,p\rangle \; \forall p\in \R^n \right. \right\} = \bigcap_{i\in I(\bar x)}\partial f_i(\bar x).
\end{align*}
\end{proof}

\begin{proposition}\label{prop:submax} Let $g:X\to \R$ be a pointwise maximum of a finite number of $C^1(X)$ functions, i.e. $g(x) = \max_{j\in J} g_j(x)$, $g_j\in C^1(X)$, and $|J|<\infty$.  Then $g$ has a nonempty Fr\'echet subdifferential that can be expressed explicitly as
\begin{equation}\label{eq:sub-cograd}
\partial g(x) = \co_{i\in J(x)} \{\nabla g_j(x)\},
\end{equation}
where $J(x)$ is the active index set. Moreover, the function $g$ is Hadamard directionally differentiable with
$$
g'(x;p) = \max_{j\in J(x)} g'_j(x;p) = \max_{v\in \partial g(x)} \langle v, p\rangle = \max_{j \in J(x)} \langle \nabla g_j(x), p\rangle \qquad \forall p\in \R^n.
$$
\end{proposition}
\begin{proof} This result is well known, and its proof can be easily deduced from the fact that the Hadamard directional derivative is the support function of the convex hull in \eqref{eq:sub-cograd}.
\end{proof}

\begin{proposition}\label{prop:epsilon-cone} Let $f =\min_{i} f_i$, where $f_i$ are Hadamard directionally differentiable with sublinear directional derivatives at $\bar x\in X$. Then for every $p\in \Sph$ there exists $\varepsilon= \varepsilon(\bar x, p)>0$ such that
$$
I(x) \subseteq I (\bar x,p) \quad \forall x = \bar x + t p + t \varepsilon u, \quad t\in (0,\varepsilon], u \in \Ball.
$$
\end{proposition}
\begin{proof} Suppose that the claim is not true. Then there exist sequences $\{\varepsilon_k\}$, $\{t_k\}$ and $\{u_k\}$ such that $\varepsilon_k\downarrow 0$,  $t_k\in (0, \varepsilon_k]$, $u_k\in \Ball$ and for
$$
x_k =  \bar x + t_k (p + \varepsilon_k u_k)
$$
we have $I(x_k)\setminus I(\bar x, p) \neq \emptyset$. Without loss of generality assume that there is an $i_0\in I$ such that $i_0 \in I(x_k)\setminus I(\bar x, p) $. Observe that
$$
f_{i_0}(x_k)=f(x_k),
$$
hence, by the continuity of $f$, $i_0\in I(\bar x)$; moreover, observing that $p + \varepsilon_ku_k \to p$, we have
$$
f_{i_0}'(\bar x;p) =  \lim_{k\to \infty} \frac{f_{i_0}(x_k)-f_{i_0}(\bar x)}{t_k}
= \lim_{k\to \infty} \frac{f(x_k)-f(\bar x)}{t_k}  = f'(\bar x;p).
$$
We then have from Proposition~\ref{prop:ddmin}
$$
\max_{v\in \partial f_{i_0}(\bar x)}\langle v,p\rangle = f'_{i_0}(\bar x;p) = f'(\bar x;p) = \min_{i\in I(\bar x)} f_i'(\bar x;p) = \min_{i\in I(x)}\max_{v\in \partial f_i(\bar x)}\langle v,p\rangle,
$$
hence, $i_0\in I(\bar x,p)$, which contradicts our assumption.
\end{proof}

\section{Limiting subdifferential for pointwise minima}

Our results rely on the following technical lemma, whose proof is inspired by the proofs of fuzzy mean value theorems for Fr\'echet subdifferential (see \cite{Ioffe,Mord}). To show the existence of a nearby point with a desired subgradient, an auxiliary function is constructed which attains a local minimum at such point.

\begin{lemma}\label{lem:main-technical} Let $f:X \to \R$ be a pointwise minimum of finitely many functions with sublinear Hadamard directional derivatives at $\bar x\in X$, as in \eqref{eq:min-function}.
Then for every $p\in \Sph$ and
\begin{equation}\label{eq:main-tech-incl}
y\in \bigcap_{i\in I(\bar x, p)} \Argmax_{v\in \partial f_i(\bar x)}\langle v,p\rangle
\end{equation}
there exist sequences $\{x_k\}
$and $\{y_k\}$ such that
	$$
	x_k\underset{k\to \infty}{\longrightarrow} \bar x, \quad \frac{x_k - \bar x}{\|x_k - \bar x\|}  \underset{k\to \infty}{\longrightarrow} p, \quad y_k \in \partial f(x_k), \quad y_k \underset{k\to \infty}{\longrightarrow} y.
	$$
\end{lemma}
\begin{proof} Fix  $p\in \Sph$ and $y$ such that
	$$
y\in \bigcap_{i\in I(\bar x, p)} \Argmax_{v\in \partial f_i(\bar x)}\langle v,p\rangle.
	$$
Observe that by the relation \eqref{eq:FrechHadSupp} and by the positive homogeneity of the directional derivative \eqref{eq:HadDDHomogeneous} we have for all $i\in I(\bar x,p)$
\begin{equation}\label{eq:ddatp}
	f'_i(\bar x;\lambda p) = \lambda f'_i(\bar x;p) = \lambda \max_{v\in \partial f_i(\bar x)} \langle v, p\rangle = \langle y, \lambda p\rangle \quad \forall\,  \lambda>0.
\end{equation}

For any $\lambda>0$ define the function $\varphi_\lambda: X\to \R$ as follows
	$$
	\varphi_{\lambda}(x)  = f(x)- f(\bar x)-\langle y, x-\bar x\rangle + \frac{1}{\lambda}\|x-(\bar x +\lambda p )\|^2.
	$$
	We will show that for sufficiently small $\lambda$ a minimum of the function $\varphi_\lambda$ on the ball $\bar x +\lambda p +\lambda \varepsilon \Ball$ is attained at an interior point (here $\varepsilon \in (0,\min(\varepsilon(\bar x,p),1))$, where $\varepsilon(\bar x,p)$ comes from Proposition~\ref{prop:epsilon-cone}). Note here that since $X$ is an open set, there exists $r\in (0,1)$ such that $\bar x + r\Ball \subset X$. If $\lambda$ is smaller than $r/2$, then $\bar x +\lambda p +\lambda \varepsilon \Ball \subset \bar x + r\Ball \subset X$. We will assume that our $\lambda$ is always chosen small enough to satisfy this condition, and also that $\lambda<\varepsilon(\bar x, p)$ (see Proposition~\ref{prop:epsilon-cone}).
	
Observe that the function $\varphi_\lambda$ is continuous, and hence it attains its minimum on the ball $\bar x +\lambda p +\lambda \varepsilon \Ball$. Assume that contrary to what we want to prove,
%there exist sequences $\{\lambda_k\}$ and $\{u_k\}$ such that this minimum is attained on the boundary, i.e.
there exist $\{\lambda_k\}$ and $\{u_k\}$ such that $\lambda_k \downarrow 0$, $u_k \in \Sph$ and 
	$$
	\bar x + \lambda_k p +\lambda_k \varepsilon u_k \in \Argmin_{x\in\bar x + \lambda_k p +\lambda_k \varepsilon \Ball} \varphi_{\lambda_k}(x).
	$$
	We therefore have
	$$
	\varphi_{\lambda_k} (\bar x + \lambda_k p + \lambda _k \varepsilon u_k )\leq \varphi_{\lambda_k} (\bar x + \lambda_k p),
	$$
	or explicitly
\begin{align}\label{eq:ineq-many}
 & \min_{i\in I}f_i(\bar x + \lambda_k p + \lambda _k \varepsilon u_k )- f(\bar x)- \langle y,  \lambda_k p + \lambda _k \varepsilon u_k \rangle + \frac{1}{\lambda_k}\| \varepsilon \lambda_k u_k \|^2 \notag\\
& \qquad 	\leq \min_{i\in I}f_i(\bar x + \lambda_k p)- f(\bar x)- \langle y, \lambda_k p\rangle .
\end{align}
By our choice of $\varepsilon$ and $\lambda$, Proposition~\ref{prop:epsilon-cone} yields that
$$
I(\bar x + \lambda_k p + \lambda _k \varepsilon u_k ) \subseteq I(\bar x,p) \quad \forall k \in \N.
$$
Without loss of generality, due to the finiteness of $I(\bar x, p)$, we can assume that the index set is constant, i.e.
$$
I(\bar x + \lambda_k p + \lambda _k \varepsilon u_k ) = \tilde I .
$$
which together with \eqref{eq:ineq-many} yields
\begin{align}\label{eq:ineq-many-individual}
	& f_i(\bar x + \lambda_k p + \lambda _k \varepsilon u_k )- f_i(\bar x)- \langle y,  \lambda_k p + \lambda _k \varepsilon u_k \rangle + \frac{1}{\lambda_k}\| \varepsilon \lambda_k u_k \|^2 \notag\\
	& \qquad 	\leq f_i(\bar x + \lambda_k p)- f_i(\bar x)- \langle y, \lambda_k p\rangle \quad \forall i \in \tilde I.
\end{align}
Notice that by our choice of $y$ we have from the relations \eqref{eq:FrechHadSupp} and \eqref{eq:ddatp} for all $i\in \tilde I \subset I(\bar x)$
\begin{equation}\label{eq:pf09}
\langle y,  \lambda_k p + \lambda _k \varepsilon u_k \rangle \leq \max_{v\in \partial f_i(\bar x)} \langle v, \lambda_k p + \lambda _k \varepsilon u_k \rangle \leq f'_i(\bar x; \lambda_k p + \lambda _k \varepsilon u_k)
\end{equation}
and
\begin{equation}\label{eq:pf10}
\langle y,  \lambda_k p \rangle =  \max_{v\in \partial f_i(\bar x)} \langle v, \lambda_k p \rangle = f'_i(\bar x; \lambda_k p)
\end{equation}
Noticing that $\|u_k\|=1$, substituting  \eqref{eq:pf09}  and \eqref{eq:pf10} into \eqref{eq:ineq-many-individual}, dividing the whole expression by $\lambda_k$, we obtain for every $i\in \tilde I$
	\begin{align*}
		\varepsilon^2 & \leq
		\frac{f_i(\bar x + \lambda_k p)- f_i(\bar x)- f_i'(\bar x; \lambda_k p)}{\lambda_k}
		-\frac{f_i(\bar x + \lambda_k p + \lambda _k \varepsilon u_k )- f_i(\bar x)- f_i'(\bar x; \lambda_k p +\lambda _k \varepsilon u_k)}{\lambda_k}\\
		& =  \frac{\oh_i(\lambda_k p )}{\lambda_k} -\frac{\oh_i(\lambda_k p + \lambda _k \varepsilon u_k)}{\lambda_k}  = \frac{\oh_i(\lambda_k p )}{\|\lambda_k p\|}- \frac{\oh_i(\lambda_k p + \lambda _k \varepsilon u_k)}{\|\lambda_k p + \lambda _k \varepsilon u_k\|}\|p + \varepsilon u_k\|
	\end{align*}
	where $\oh_i (\cdot)$'s are as in  \eqref{eq:HadPropIndividual}. It is not difficult to see that the right hand side goes to zero as $k\to \infty$, and hence $\varepsilon^2=0$, which contradicts our choice of a fixed positive $\varepsilon$.
	
	We have shown that our assumption is wrong, and given a fixed $\varepsilon >0$ for sufficiently small $\lambda(\varepsilon)$ the function $\varphi_\lambda$ has a local minimum in the interior of the ball $\bar x + \lambda p + \varepsilon \lambda \Ball$ for all $\lambda \in (0,\lambda(\varepsilon))$; in other words, it attains an unconstrained local minimum at this point.
	
	Let $\{\varepsilon_k\}$ be such that $\varepsilon_k \downarrow 0$, and choose
	$$
	\lambda_k = \min\{\varepsilon_k, \lambda(\varepsilon_k)\} \quad \forall \, k \in \N.
	$$
	For each $k\in \N$ there exists a point $u_k$ in the interior of $\Ball$ such that $\bar x + \lambda_k p + \varepsilon_k \lambda_ k u_k$ is a minimum of the function $\varphi_\lambda$ on the ball $\bar x + \lambda_k p + \varepsilon_k \lambda_k \Ball$.
From the optimality condition in Proposition~\ref{prop:optimality} we have
	$$
	y_k: = y- 2 \varepsilon_k u_k \in \partial f (\bar x + \lambda_k p + \lambda_k \varepsilon_k u_k).
	$$
	Observe that $y_k \to y$ and for $x_k : = \bar x + \lambda_k p + \lambda_k \varepsilon_k u_k$ we have
	$x_k \to \bar x$ and
	$$
	\frac{x_k - \bar x}{\|x_k - \bar x\|} =\frac{p + \varepsilon_k u_k}{\|p + \varepsilon_k u_k\|}    \underset{k\to \infty}{\longrightarrow} p,
	$$
	so we are done.
\end{proof}

We use Lemma~\ref{lem:main-technical} to obtain an inclusion relation for the outer limits of subdifferentials.

\begin{theorem}\label{con:gen3.2} Let $f:X 	\to \R$ be as in \eqref{eq:min-function}. Then
\begin{equation}\label{eq:3.2main}
	\bigcup_{p\in \Sph \atop f'(\bar x; p)>0} \bigcap_{i\in I(x,p)} \Argmax_{v\in \partial f_i(\bar x)}\langle v, p\rangle \subseteq \Limsup_{x\to \bar x\atop f(x) >f(\bar x)} \partial f(x),
\end{equation}
where
$I(x,p)$ is the index set as defined in \eqref{eq:Ip}.	
\end{theorem}
\begin{proof} For $\bar x\in X$ choose any direction $p\in \Sph$ such that $f'(\bar x;p)>0$. We will show that for
	$$
	y \in  \bigcap_{i\in I(\bar x,p)} \Argmax_{v\in \partial f_i(\bar x)}\langle v, p\rangle
	$$
	we have
	$$
	y \in \Limsup_{x\to \bar x\atop f(x) >f(\bar x)}  \partial f(x).
	$$
	By Lemma~\ref{lem:main-technical} there exist sequences $\{x_k\}$ and $\{y_k\}$ such that $y_k \in \partial f(x_k)$, $y_k \to y$, $x_k\to \bar x $ and
	$$
	p_k := \frac{x_k - \bar x}{\|x_k - \bar x\|}\to p.
	$$
	It remains to show that for sufficiently large $k$ we have $f(x_k)>f(\bar x)$. Assume this is not so. Then without loss of generality, $f(x_k)\leq f(\bar x)$ for all $k\in \N$, and
	$$
	f'(\bar x;p) = \lim_{k\to \infty}\frac{f(\bar x+ \|x_k-\bar x\|p_k)- f(\bar x)}{\|x_k - \bar x \|}= \lim_{k\to \infty}\frac{f(x_k)- f(\bar x)}{\|x_k - \bar x \|} \leq 0,
	$$
	which contradicts our choice of $p$.
\end{proof}

We have the following corollary that follows directly from Theorem~\ref{con:gen3.2} and the fact that the limit on the right-hand side of \eqref{eq:3.2main} is a closed set.

\begin{corollary}\label{cor:genth3.2} Let $f:X 	\to \R$ be as in \eqref{eq:min-function}. Then
\begin{equation}\label{eq:3.2cormain}
	\cl \left(\bigcup_{p\in \Sph \atop f'(\bar x; p)>0} \bigcap_{i\in I(x,p)} \Argmax_{v\in \partial f_i(\bar x)}\langle v, p\rangle \right)\subseteq \Limsup_{x\to \bar x\atop f(x) >f(\bar x)} \partial f(x),
\end{equation}
where
$I(x,p)$ is the index set as defined earlier.	
\end{corollary}

We would like to point out that the result of the computation of the expression on the left hand side of \eqref{eq:3.2main} and \eqref{eq:3.2cormain} depends on the position of zero with respect to the subdifferential. The following example taken from \cite[Example~3.2]{RoshchinaMord} illustrates this observation.

\begin{example}  Consider the two functions
$$
 f_1(x,y) = \sqrt{x^2+ y^2} +\frac{1}{2}x,\quad  f_2(x,y) = \sqrt{x^2+ y^2} -\frac{1}{2}x.
$$
Notice that $f = \min \{f_1, f_2\}$ is nonnegative everywhere, and $f(x,y) = 0$ iff $(x,y)=0_2$ (see the plots in Fig.~\ref{Fig:disks-plots}).
\begin{figure}[ht]
{\centering \includegraphics[scale=1]{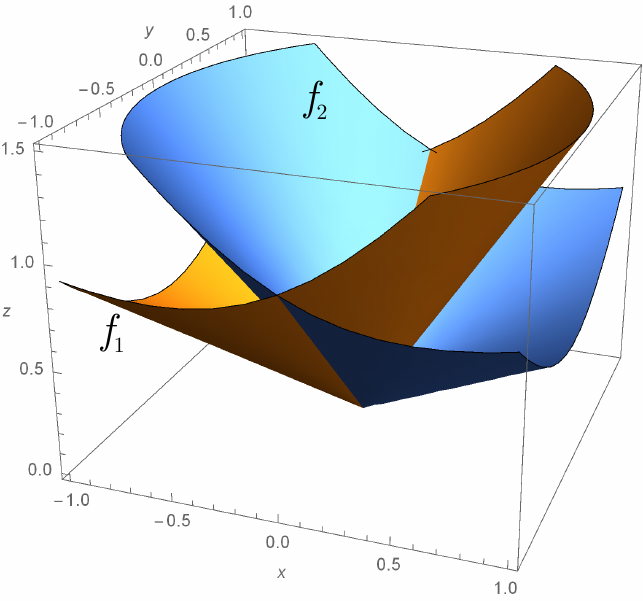} \quad
\includegraphics[scale=1]{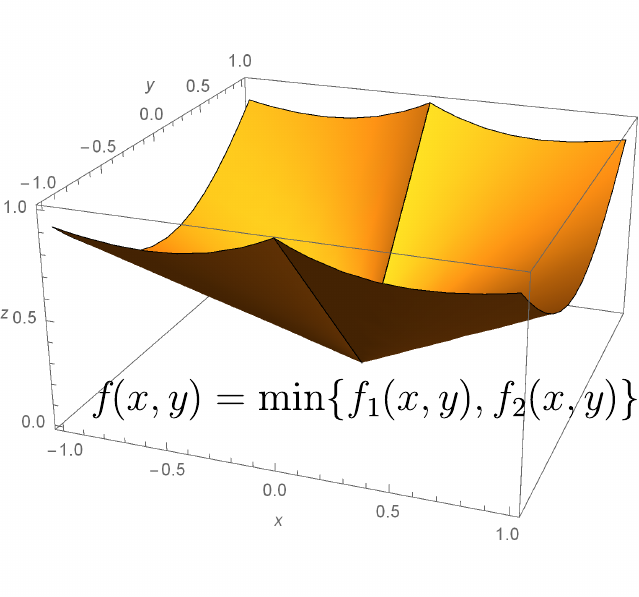} \\}
\caption{On the left: the functions $f_1$ and $f_2$; on the right: $f(x,y) = \min \{f_1(x,y), f_2(x,y)\}$.}
\label{Fig:disks-plots}
\end{figure}
It is not difficult to observe that the subdifferentials of $f_1$ and $f_2$ at zero are unit disks centred at $(\frac{1}{2}, 0)$ and $(-\frac{1}{2},0)$ respectively (see Example~3.2 in \cite{FrechetCalculus} for detailed explanation). The left-hand side in \eqref{eq:3.2main} for the function $f$ at $0_2$ is the union of two (open) semi-circles, see Fig.~\ref{Fig:disks-plots-sub}. Observe that the closure of this set coincides with the outer limit on the right hand side of \eqref{eq:3.2main}, so in fact \eqref{eq:3.2cormain} is an exact characterisation in this case.
\begin{figure}[ht]
{\centering \includegraphics[height=180pt]{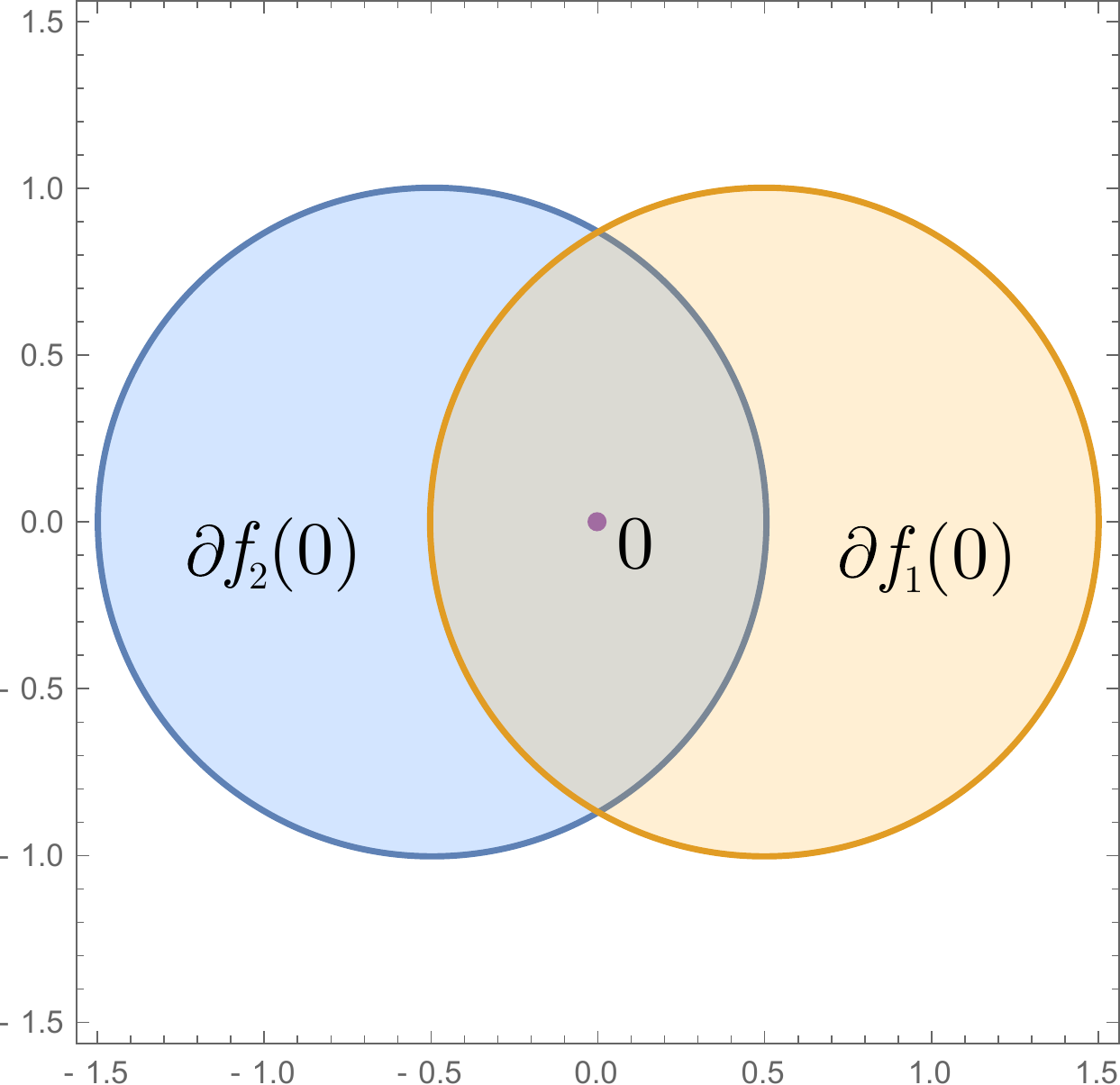}\qquad
\includegraphics[height=180pt]{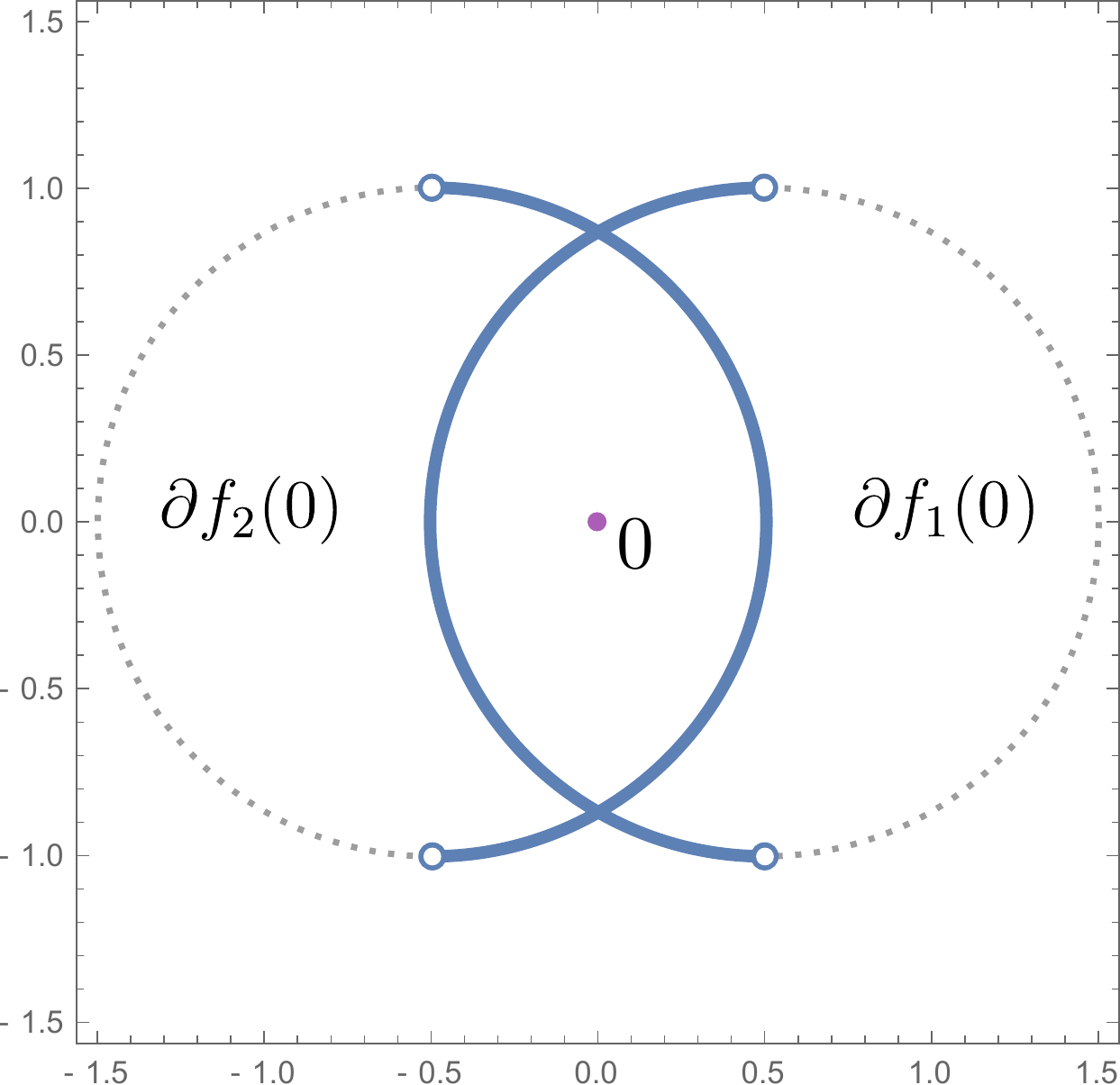}\\}
\caption{On the left: the Fr\'echet subdifferentials at zero, $\partial f_1 (0)$ and $\partial f_2(0)$; on the right: the left hand side union in \eqref{eq:3.2main} is shown in bold solid lines.}
\label{Fig:disks-plots-sub}
\end{figure}

We next modify this example by translating the subdifferentials and obtaining a different set on the left hand side of \eqref{eq:3.2main}.

Consider the modified functions
$$
 \tilde f_1(x,y) = \sqrt{x^2+ y^2} +\frac{3}{2}x,\quad  \tilde f_2(x,y) = \sqrt{x^2+ y^2} +\frac{1}{2}x.
$$
The minimum function $\tilde f(x,y) = \min \{\tilde f_1(x,y) , \tilde f_2(x,y)\}$ is no longer nonnegative (see Fig.~\ref{Fig:disks-mod-plots}.
\begin{figure}[ht]
{\centering \includegraphics[height=120pt]{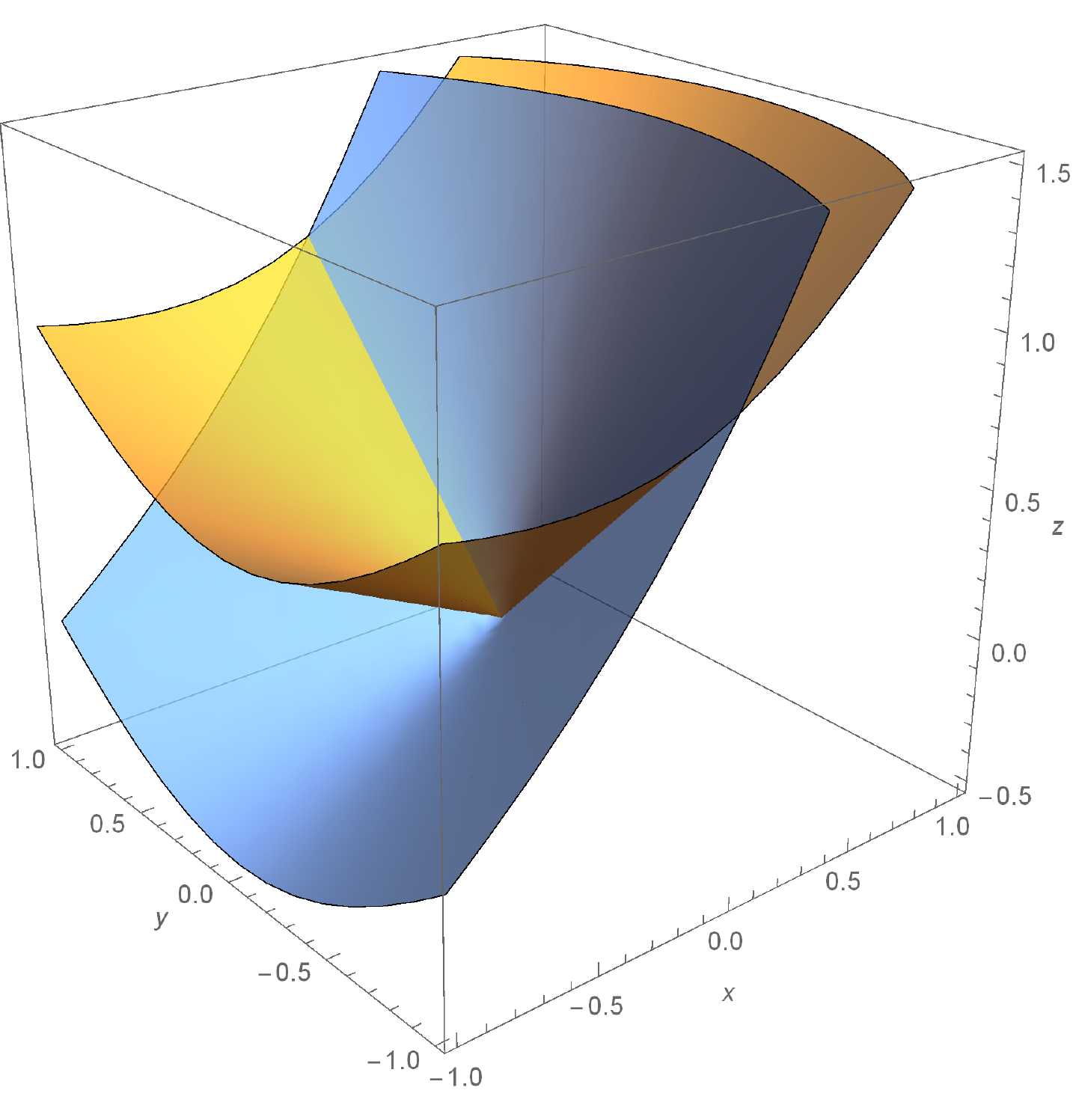} \quad
\includegraphics[height=120pt]{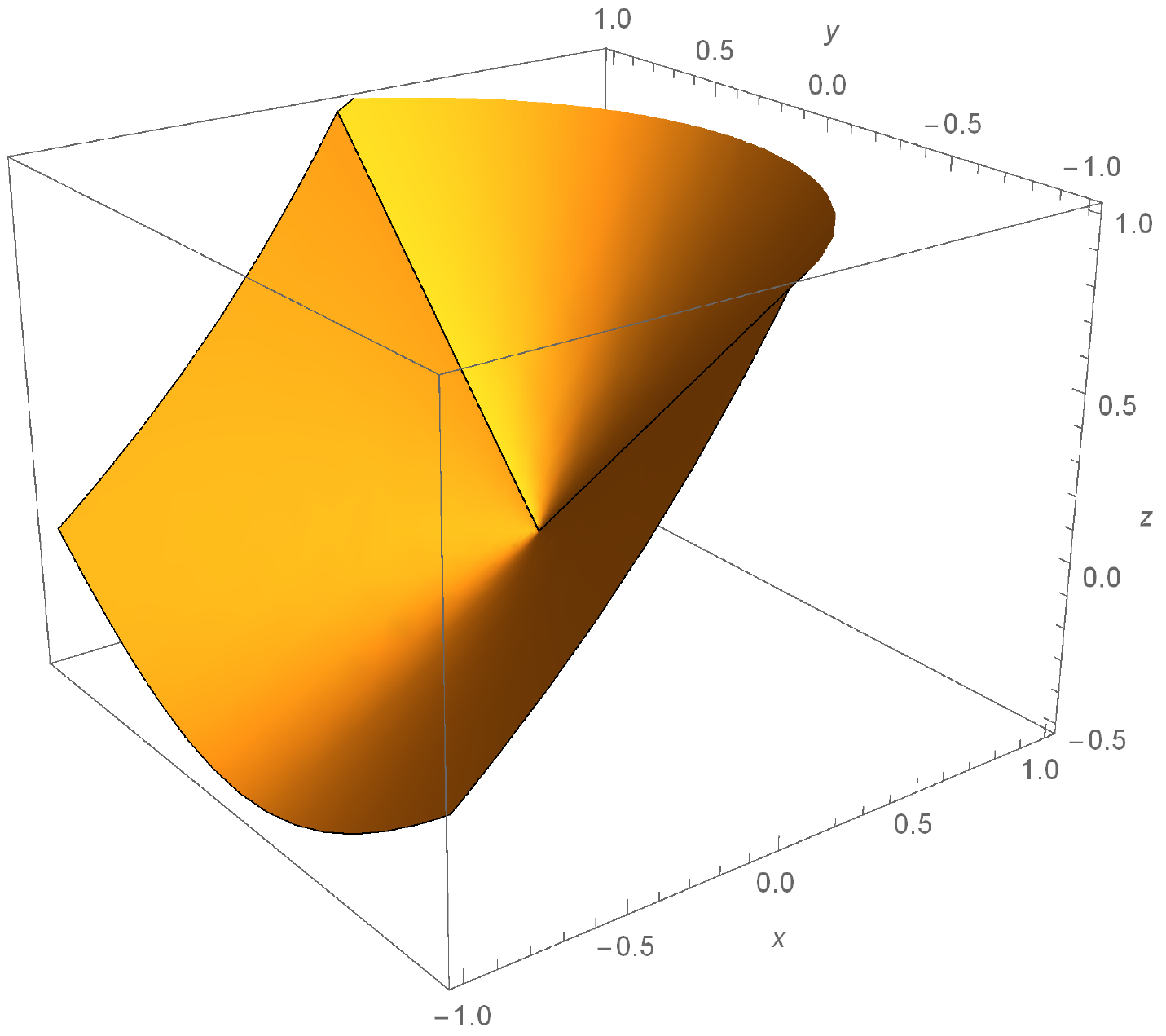}\quad
\includegraphics[height=120pt]{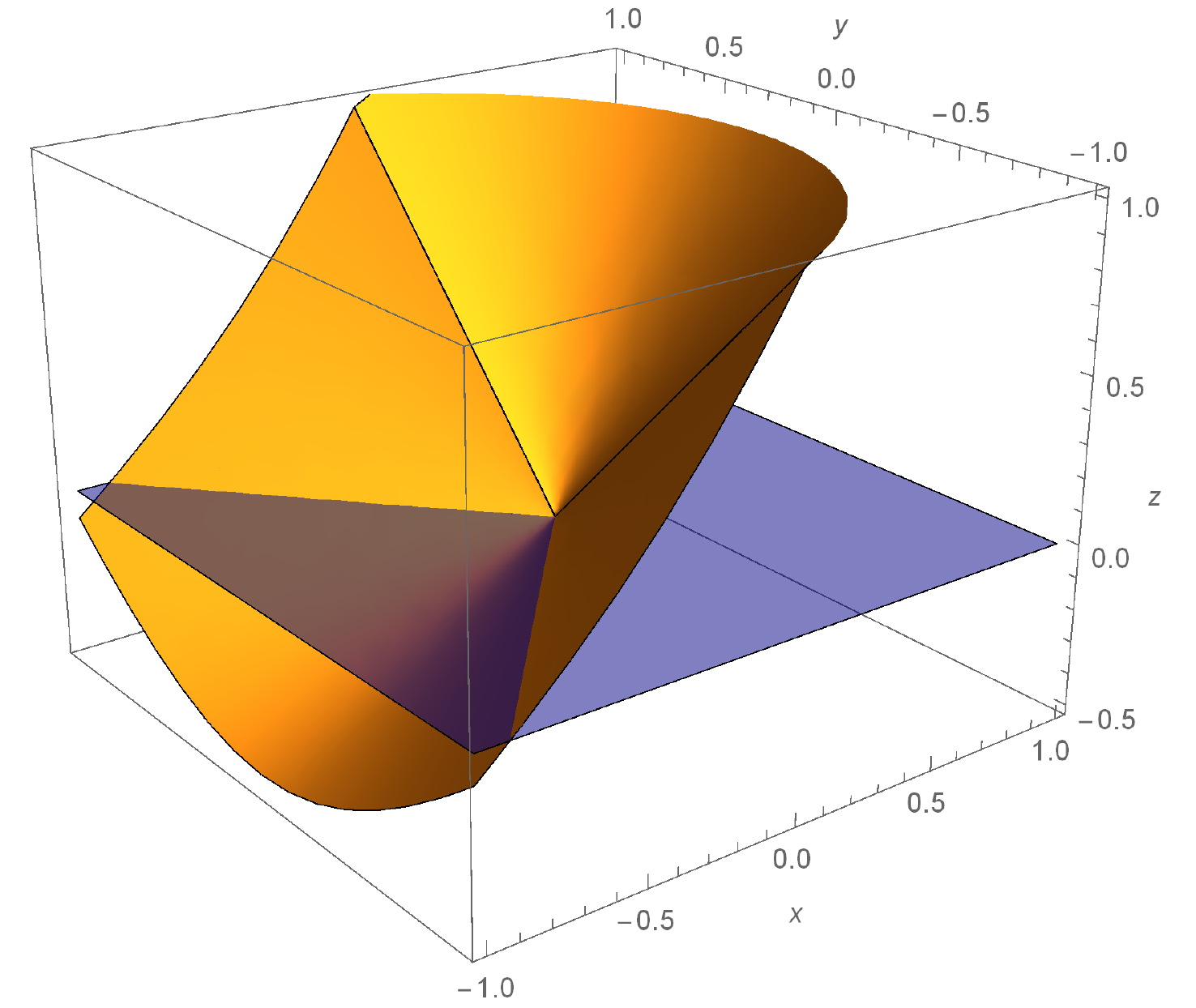} \\}
\caption{On the left: the functions $\tilde f_1$ and $\tilde f_2$; in the middle: $\tilde f(x,y) = \min \{\tilde f_1(x,y), \tilde f_2(x,y)\}$; on the right: the plot of $\tilde f$ is shown together with $\{z=0\}$.}
\label{Fig:disks-mod-plots}
\end{figure}
Similar to the previous example, the subdifferentials of $f_1$ and $f_2$ at zero are unit disks centred at $(\frac{3}{2}, 0)$ and $(\frac{1}{2},0)$ respectively. The left-hand side in \eqref{eq:3.2main} is the union of one semi-circle and two smaller segments of the other circle, see Fig.~\ref{Fig:disks-mod-plots-sub}.
\begin{figure}[ht]
{\centering \includegraphics[height=180pt]{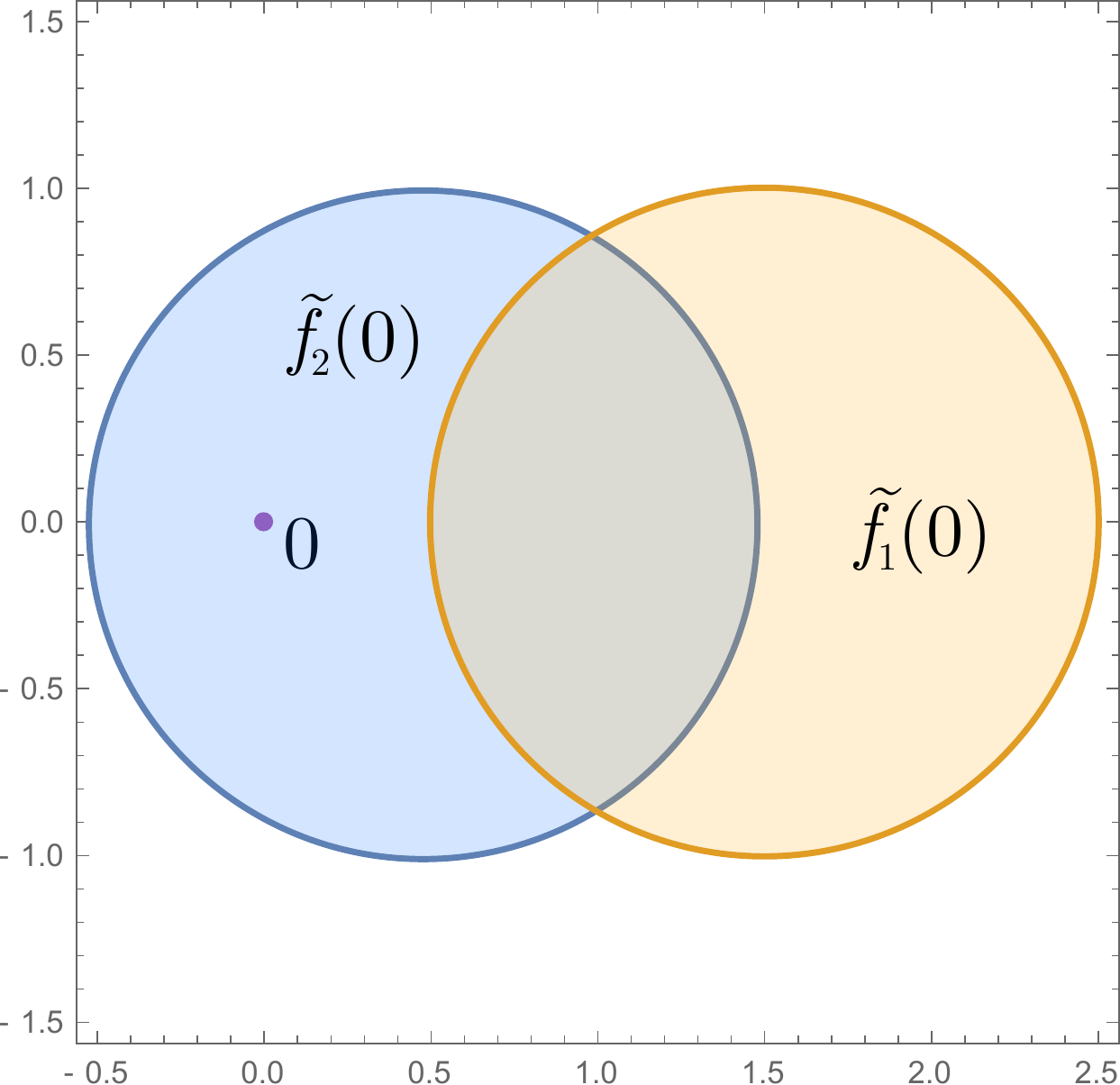}\qquad
\includegraphics[height=180pt]{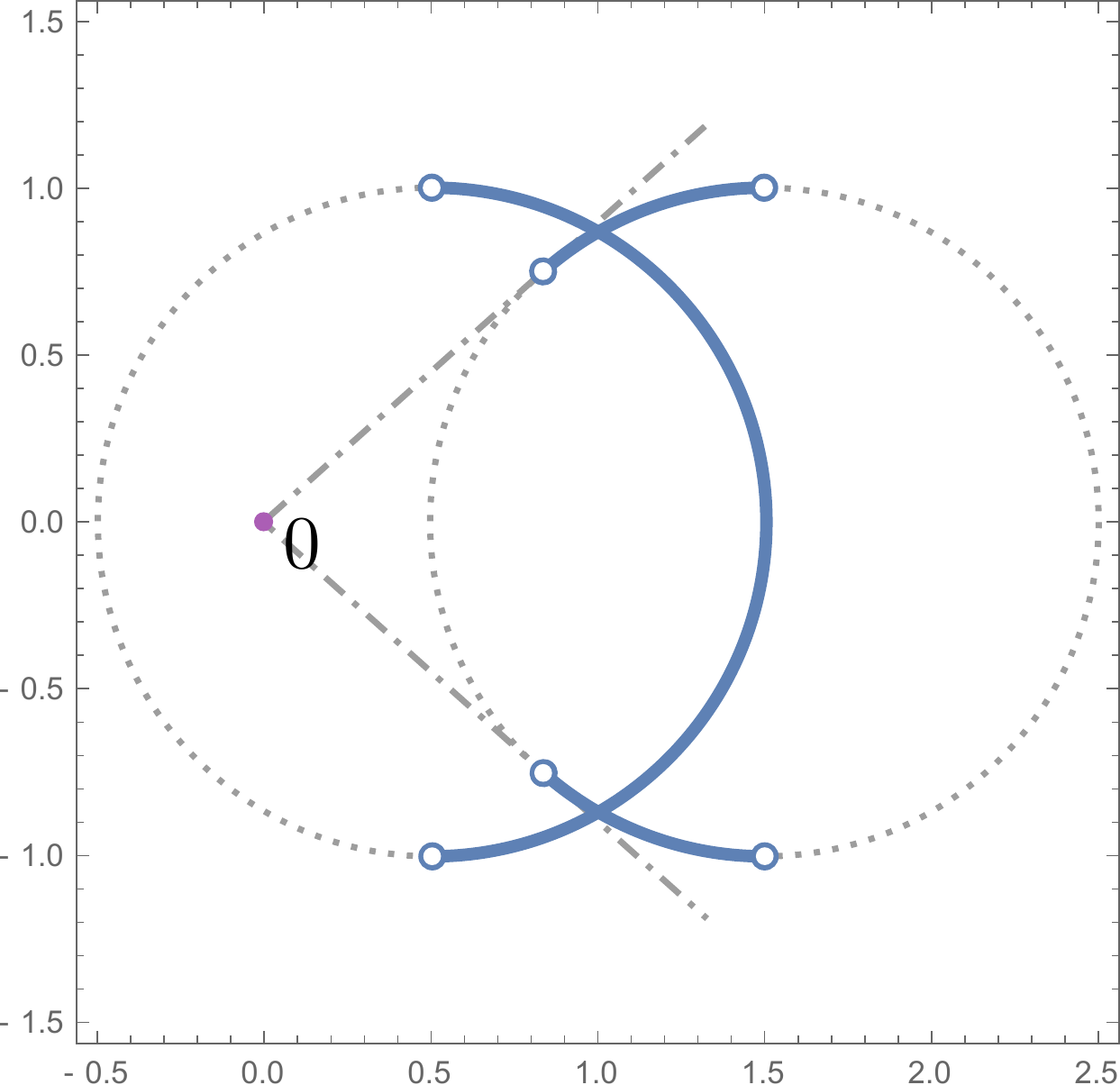}\\}
\caption{On the left: the Fr\'echet subdifferentials at zero, $\partial f_1 (0)$ and $\partial f_2(0)$; on the right: the left hand side union in \eqref{eq:3.2main} is shown in bold solid lines.}
\label{Fig:disks-mod-plots-sub}
\end{figure}
\end{example}

We have the following useful special case of Theorem~\ref{con:gen3.2}.

\begin{corollary}\label{cor:gen3.2} Let $f:X \to \R$ be Hadamard directionally differentiable at $\bar x\in X$, assume that the directional derivative $f'(\bar x;\cdot)$ is a sublinear function, then
\begin{equation}\label{eq:corspecial}
	\cl \left(\bigcup_{p\in \Sph \atop f'(\bar x; p)>0} \Argmax_{v\in \partial f(\bar x)}\langle v, p\rangle\right) \subseteq \Limsup_{x\to \bar x\atop f(x) >f(\bar x)} \partial f(x).
\end{equation}
\end{corollary}

\begin{remark}\label{rem:KLY} Observe that in the notation of \cite{KaiwenMinghua} the closure of the union on the left hand side of \eqref{eq:corspecial} coincides with the outer limit of the Fr\'echet subdifferentials of the support of $\partial f(\bar x)$. Hence we answer affirmatively the open question of \cite{KaiwenMinghua} on whether such outer limit is a subset of the right hand side of \eqref{eq:corspecial}.
\end{remark}

In Corollary~\ref{cor:gen3.2lin} we strengthen Theorem~3.2 of \cite{OuterLimits}, dropping the affine independence assumption. We first recall the notation from \cite{OuterLimits} and the related geometric constructions.

Let $g:X\to \R$ be a pointwise maximum of smooth functions, i.e.
\begin{equation}\label{eq:defg}
g (x) = \max_{j\in J} g_j(x), \quad  g_j \in C^1(X) \quad \forall j \in J,
\end{equation}
where $J$ is a finite index set. As in \cite{OuterLimits} define the collection $\D(\bar x)$ of index subsets $D\subset J(\bar x)$ such that the following system is consistent with respect to $d$
\begin{equation}\label{eq:consistent-system}
\left\{
\begin{array}{lr}
\langle \nabla g_j(\bar x), d\rangle = 1, &   j \in D,\\
\langle \nabla g_j(\bar x),d \rangle <1, & j \in J(\bar x)\setminus D.
\end{array}	
\right\}
\end{equation}

\begin{corollary}\label{cor:gen3.2lin}
 Let $g(x)$ to be the pointwise maximum of smooth functions as in \eqref{eq:defg}. Then
	\begin{equation}
	\bigcup_{D\in\D(\bar{x})}\mathrm{conv}\left\{  \nabla g_{j}(\bar
	{x}),\;j\in D\right\}  =\bigcup_{p\in \Sph ,g^{\prime}(\bar{x}
		,p)>0}\Argmax_{v\in\partial g(\bar{x})}\langle v,p\rangle
	\subseteq
	\Limsup_{
		\genfrac{}{}{0pt}{}{x\rightarrow\bar{x}}{g(x)>g(\bar{x})}
	}\partial g(x), \label{eqn:123}
	\end{equation}
	in other words, in \cite[Theorem~3.2]{OuterLimits} the subsets $\mathcal{D}_{AI}(\bar
	{x})$ can be replaced by $\mathcal{D}(\bar{x})$.
	
	Moreover when all $\{g_j\}_{j\in J}$ are affine we have an identity (instead of an inclusion)
	in (\ref{eqn:123}).
\end{corollary}

\begin{proof}
We begin by showing  the following identity:
	\begin{equation}
	\displaystyle\bigcup_{D\in\mathcal{D}(\bar{x})}\co\{\nabla g_{i}(\bar{x}),i\in
	D\}=\bigcup_{p\in \Sph ,g^{\prime}(\bar{x},p)>0}\Argmax_{v\in\partial
		g(\bar{x})}\langle v,p\rangle. \label{eq:pf33445_new}
	\end{equation}
	Observe that explicitly the equality below holds for all $p$ (see Proposition~\ref{prop:submax})
	\[
	\partial g(x)=\co\{\nabla g_{i}(x)\,|\,i\in J(x)\},\qquad g^{\prime}
	(x;p)=\max_{v\in\partial g(x)}\langle p,v\rangle \quad \forall p,
	\]
	hence, for each direction $p$ that features in the union on the right hand
	side of \eqref{eq:pf33445_new} the relevant $\Argmax$ gives the support face of 	the subdifferential.
%	\[
%	\partial g(\bar{x})=\co\{\nabla g_{i}(\bar{x})\,|\,i\in J(x)\}.
%	\]
	Explicitly, fix $p\in \Sph$ and let
	\[
	s(p):=g^{\prime}(x;p)=\max_{v\in\partial g(\bar{x})}\langle v,p\rangle
	=\max_{j\in J(\bar{x})}\langle\nabla g_{j}(\bar{x}),p\rangle,
	\]
	then
	\begin{align*}
	\bigcup_{p\in \Sph \atop g^{\prime}(\bar{x},p)>0}
	\Argmax_{v\in\partial g(\bar{x})}\langle v,p\rangle %& =\bigcup_{	\genfrac{}{}{0pt}{}{p\in \Sph}{s(p)>0}%
%}\Argmax_{j\in J(\bar{x})}\langle\nabla g_{j}(\bar{x}),p\rangle\\
& =\bigcup_{p\in \Sph\atop s(p)>0}\Argmax_{v\in\co_{j\in J(\bar{x})}\{\nabla g_{j}(\bar{x})\}}\langle
v,p\rangle.
\end{align*}
We now get back to the definition of our index subsets $\mathcal{D}$. The system \eqref{eq:consistent-system}
%\[
%\left\{
%\begin{array}
%[c]{lr}%
%\langle\nabla g_{j}(\bar{x}), d\rangle=1, & j\in D,\\
%\langle\nabla g_{j}(\bar{x}), d\rangle<1, & j\in J(\bar{x})\setminus D.
%\end{array}
%\right\}
%\]
is consistent for some nonempty $D\subset J(\bar x)$ and $d\in\R^{n}\setminus \{0\}$ if and only if for $p=d/\|d\|$
$$
g'(\bar x; p) = \frac{1}{\|d\|} g'(\bar x;d) = \frac{1}{\|d\|}\max_{j\in J(\bar x)}\langle\nabla g_{j}(\bar{x}), d\rangle = \frac{1}{\|d\|}>0,
$$
and
$$
\Argmax_{v\in \co\{\nabla g_{i}(\bar x)\,|\,i\in J(\bar x)\} } \langle v,p\rangle = \{\nabla g_i(\bar x) \, |\, i\in D\}.
$$
hence we get \eqref{eq:pf33445_new}. The last inclusion of \eqref{eqn:123} follows from Corollary~\ref{cor:gen3.2}.

Finally, to show that in the affine case an equality holds in \eqref{eq:pf33445_new}, observe that there is a sufficiently small neighbourhood $N(\bar x)$ of $\bar x$ on which the affine function $g$ coincides with the sum $g(\bar x) + \sigma_\partial g(\bar x) (x-\bar x)$, where $\sigma_\partial g(\bar x) (\cdot)$ is the support of the subdifferential, and hence for any $x$ in this neighbourhood we have $\partial g(x) = \partial \sigma_\partial g(\bar x) (x-\bar x)$. Since the number of different subdifferentials of points in this neighbourhood is finite, the right hand side is in fact the union
$$
\Limsup_{
		\genfrac{}{}{0pt}{}{x\rightarrow\bar{x}}{g(x)>g(\bar{x})}
	}\partial g(x) = \bigcup_{x\in N(\bar x) \atop g(x)>g(\bar x)
			} \partial \sigma_{\partial g(\bar x)}(x)
			=\bigcup_{p\in \Sph \atop g'(\bar x;p)>0
						}\Argmax_{v\in\partial g(\bar{x})} \langle v, p\rangle,
$$
where the last equality follows from Proposition~\ref{prop:classic-subdiff}. Note here that $g(x)>g(\bar x)$ if and only if $\partial \sigma_\partial g(\bar x) (x-\bar x)>0$ or equivalently $g'(\bar x, x-\bar x)>0$.
\end{proof}

In the expression \eqref{eq:pf33445_new}, the $\Argmax$ construction gives the support faces of the subdifferential, while the positivity constraint on the directional derivative means that zero lies in the same (strict) half-space of some hyperplane exposing the face as the rest of the subdifferential. We consider a very simple example that demonstrates how this construction works.

\begin{example}\label{eg:marco-illustr} Let $g = \min_{1\in \{1,\dots, 6\}} g_i$, where $g_i:\R^2\to \R$ are linear functions, 
$$
g_1 (x) = 5x_1, \quad 
g_2 (x) = 2x_1+x_2, \quad
g_3 (x) = x_1, 
$$
$$
g_4 (x) = 2x_1-2x_2, \quad
g_5 (x) = 4x_1-2x_2, \quad
g_6 (x) = 3x_1-x_2.
$$
Observe that each of these functions is active at zero, so we have $J(0) = \{1,2,\dots, 6\}$. Furthermore, we have the gradients 
$$
a_1 := \nabla g_1(0) = (5,0)^T, \quad
a_2 := \nabla g_2(0) = (2,1)^T, \quad
a_3 := \nabla g_3(0) = (1,0)^T, 
$$
$$
a_4 := \nabla g_4(0) = (2,-2)^T, \quad
a_5 := \nabla g_5(0) = (4,-2)^T, \quad
a_6 := \nabla g_6(0) = (3,-1)^T.
$$
These gradients are shown in the left hand side image of Fig.~\ref{fig:marco}.
\begin{figure}[ht]
{\centering \includegraphics[scale = 0.9]{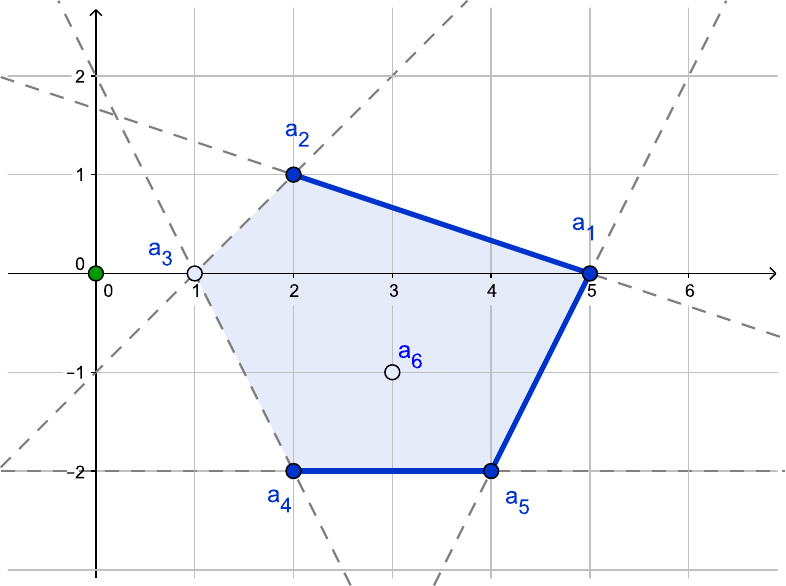} \qquad  \includegraphics[scale = 0.9]{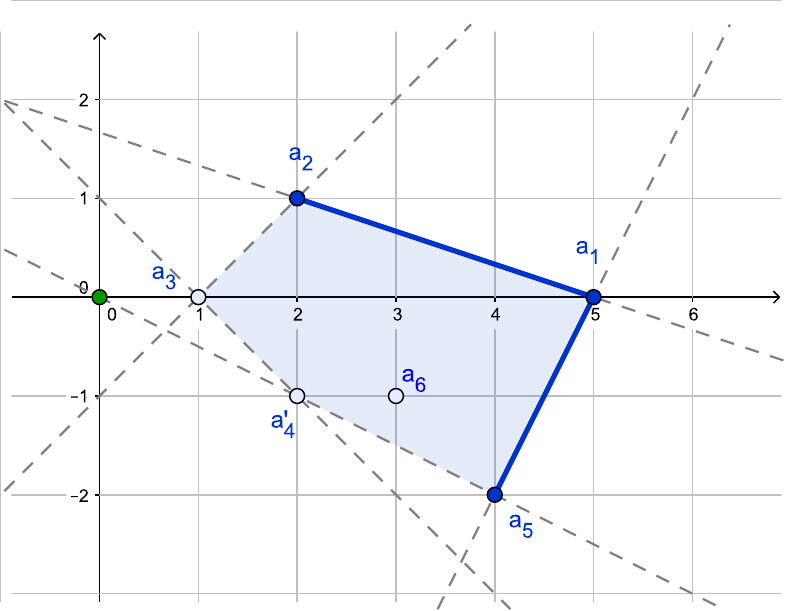} \\}
\caption{The geometric constructions for Example~\ref{eg:marco-illustr}.}
\label{fig:marco}
\end{figure}

It is not difficult to see that 
$$
\mathcal{D}(0) = \{\{1\}, \{2\},\{4\},\{5\}, \{1,2\},\{1,5\}, \{4,5\}\},
$$
the relevant set 
$$
\bigcup_{D\subset \mathcal D(0)} \co \{\nabla g_j, j\in D\} = \bigcup_{D\subset \mathcal D(0)} \co \{a_j, j\in D\}
$$
is shown in Fig.~\ref{eg:marco-illustr} with thick solid lines.

If we replace the function $g_4$ with $g'_4 = 2x_1-x_2$, and construct the relevant arrangement of the gradients $a_1, a_2, a_3, a'_4, a_5, a_6$, we would lose the index subsets $\{4\}$ and $\{4,5\}$, since now zero is on the line connecting $a_4'$ with $a_5$, and there is no $d$ that would make  \eqref{eq:consistent-system} consistent for $D=\{4\}$ and $D = \{4,5\}$ (see the image on the right-hand side of Fig.~\ref{eg:marco-illustr}). 
\end{example}

The next example is taken from \cite{KaiwenMinghua}. It shows that even for smooth functions it is not always possible to replace the inclusion in \eqref{eqn:123} with an equality.
 
\begin{example}\label{eg:sliced}Let $g:\R^2\to \R$ be defined as the piecewise maximum of two smooth functions, 
    $$f(x): = \max \{g_1(x), g_2(x)\},$$
    where
$$
g_1(x) = x_1^2 + x_2^2 + \frac{1}{2}(x_1+x_2) = \left(x_1 + \frac{1}{4}\right)^2 + \left(x_2 + \frac{1}{4}\right)^2 - \frac{1}{8}, \quad f_2(x) = x_1 + x_2.
$$   
The graph of $g$ is shown in Fig.~\ref{Fig:plot-sliced}.
\begin{figure}[ht]
{\centering \includegraphics[height=180pt]{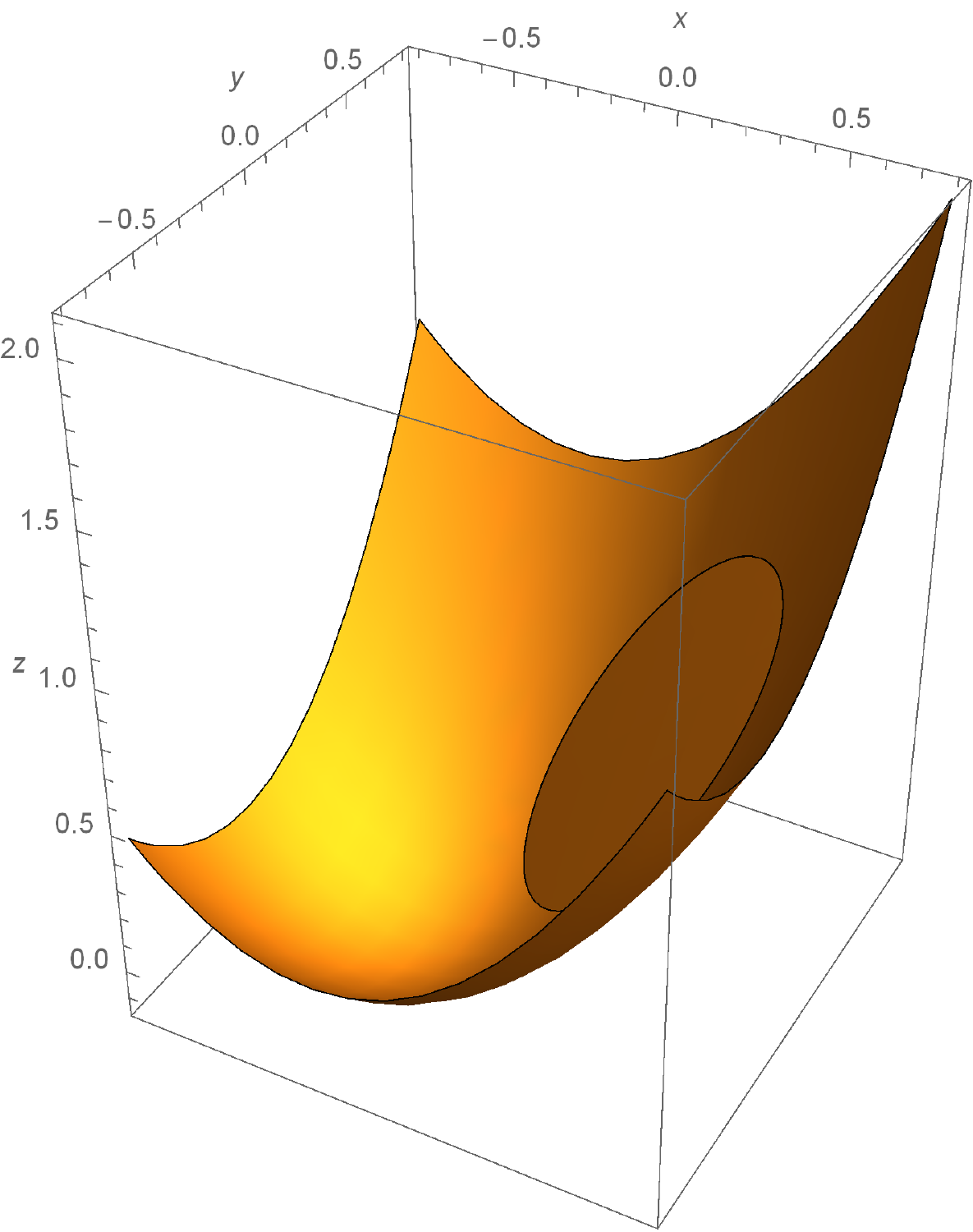}\qquad
\includegraphics[height=180pt]{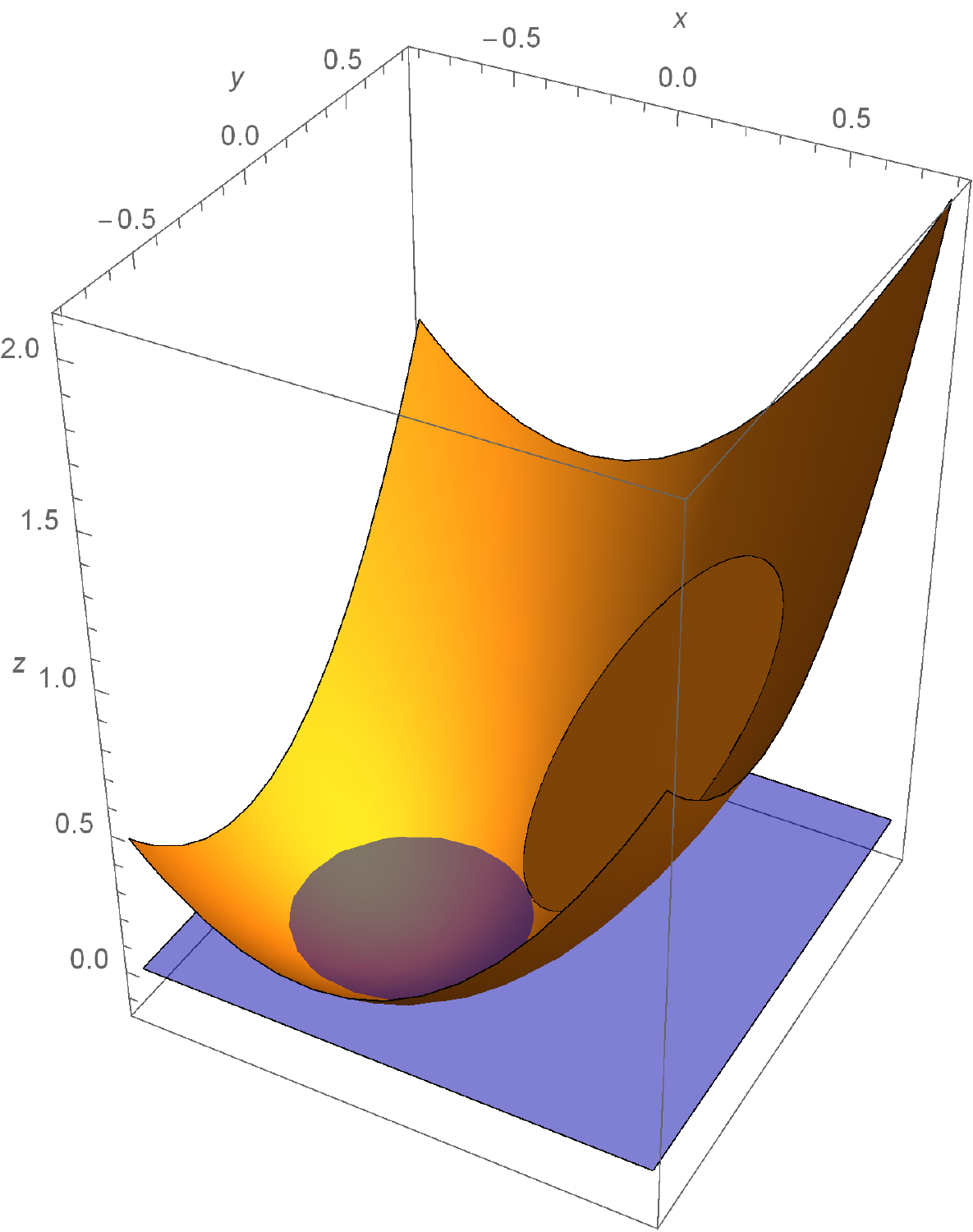}\\}
\caption{The function $\varphi$ from Example~\ref{eg:sliced}.}
\label{Fig:plot-sliced}
\end{figure}
It is clear from the illustration that the directional derivative at the point $x=0_2$, where the level set $\{x\,|\, g(x) = 0\}$ touches the linear `slice' of the graph, the Hadamard directional derivative is a piecewise linear function, however the outer limits of subdifferentials capture the liming gradients that come from the gradients of the curved parts of the graph, and this results in the underestimation of the outer limit. We have (see the explanation in \cite{KaiwenMinghua}),
$$
	\Limsup_{
		\genfrac{}{}{0pt}{}{x\rightarrow\bar{x}}{g(x)>g(0)=0}
	}\partial g(x) = \co \left\{\left(\frac 1 2 , \frac 1 2 \right)^T, \left(1  , 1  \right)^T \right\},
$$
however checking the consistency of the system \eqref{eq:consistent-system} for different index subsets, $\{1\}, \{2\}, \{1,2\}$, it is easy to see that this system has a solution $d$ only for $\{2\}$, hence, $\mathcal{D}(0) = \{2\}$, and the left-hand side of \eqref{eqn:123} gives
$$
	\bigcup_{D\in\D(0)}\mathrm{conv}\left\{  \nabla g_{j}(0),\;j\in D\right\}  =  \left\{  \nabla g_{j}(0),\;j\in \{2\}\right\} =\{ \nabla g_2(0)\} = \{\left(1  , 1  \right)^T\}.
$$
We have hence `lost' the second gradient.
%
%
%Consider that $f_1 < f_2$ if and only if
%$$x_1^2 + x_2^2 + \frac{1}{2}(x_1 + x_2) < x_1 + x_2$$
%Rearranging the expression, we get
%$$x_1^2 + x_2^2 - \frac{1}{2} < 0$$
%Then complete the square,
%$$(x_1 - \frac{1}{4})^2 + (x_2 - \frac{1}{4})^2 - \frac{1}{8} < 0$$
%Which means $f_1 < f_2$ inside of the circle
%$$(x_1 - \frac{1}{4})^2 + (x_2 - \frac{1}{4})^2 = \frac{1}{8}$$\\[+5pt]
%
%When $f = f_1$, we have
%$$\nabla f = (2x_1 + \frac{1}{2}, 2x_2 + \frac{1}{2}) \longrightarrow (\frac{1}{2}, \frac{1}{2}) \text{ as } x \rightarrow 0$$
%When $f = f_2$, we have $\nabla f = (1,1)$.\\[+3pt]
%Therefore,
%$$\partial f(0) = \conv \{(\frac{1}{2},\frac{1}{2}) (1,1)\}$$
%$$\partial^{>} f(0) = \conv \{(\frac{1}{2},\frac{1}{2} ), (1,1)\}$$
%Also, $f'(x,d) > 0 $ if and only if $d_1 + d_2 > 0$.\\[+3pt]
%Therefore,
%$$\cup_{d \in S, f'(0,d)>0} \Argmax_{v \in \partial f(0)} \langle v,d \rangle = \{(1,1)\}$$
%\begin{align*}
%d(0, \partial^{<} f(0)) & = \frac{\sqrt{2}}{2} \\
%                        & = \frac{1}{\sqrt{2}} \\
%                        & = \Er f(0)\\
%                        & < d(0, \partial^{<} \sigma_{\partial f(0)} (0))\\
%                        & = \sqrt{2}
%\end{align*}
\end{example}

\section{Exact representations for piecewise affine functions    }

We are now ready to generalise Theorem~3.1 from \cite{OuterLimits}. We first prove that for positively homogeneous functions the inclusion \eqref{eq:3.2main} can be replaced by an equality.

\begin{lemma}\label{lem:techPH} Let $h:\R^n\to \R$ be a pointwise minimum of a finite number of sublinear functions, i.e.
$$
h(x) = \min_{i\in I} h_i(x), \quad h_i (x) = \max_{v\in C_i} \langle v, x\rangle \quad \forall i \in I,
$$	
where $C_i$ is a compact convex set for each $i\in I$. Then
\begin{equation}\label{eq:3.2mainspecial}
 \cl 	\bigcup_{x\in \Sph \atop h(x)>0} \bigcap_{i\in I(x)} \Argmax_{v\in C_i}\langle v, x\rangle = \Limsup_{x\to 0\atop h(x) >0} \partial h(x).
\end{equation}
\end{lemma}
\begin{proof} Observe that the inclusion ``$\subseteq$'' in \eqref{eq:3.2mainspecial} follows directly from Theorem~\ref{con:gen3.2} substituting $\bar x = 0$, observing that $h'(0;p) = h(p)$, $h(0)=0$ (so $h(p)>0$ iff $h'(0;p)>0$) and that the right hand side is a closed set. It remains to show the reverse inclusion. Choose any
$$
y \in \Limsup_{x\to 0\atop h(x) >0} \partial h(x).
$$
There exist sequences $\{x_k\}$ and $\{y_k\}$ such that $x_k\to 0$, $y_k \to y$ and $y_k \in \partial h(x_k)$. We have by Proposition~\ref{prop:calc-intersection}
$$
\partial h(x_k) = \bigcap_{i\in I(x_k)} \partial h_i(x_k);
$$
furthermore, Proposition~\ref{prop:classic-subdiff} yields
$$
\partial h_i(x_k) = \Argmax_{v\in C_i}\langle v, x_k\rangle,
$$
and hence
\begin{equation}\label{eq:344656456}
y_k \in \partial h(x_k) = \bigcap_{i \in I(x_k)} \Argmax_{v\in C_i}\langle v, x_k\rangle.
\end{equation}
Observe that since $x_k\neq 0$ and $h_i$'s are positively homogeneous, we have 
$$
\Argmax_{v\in C_i}\langle v, x_k\rangle = \Argmax_{v\in C_i}\left\langle v, \frac{x_k}{\|x_k\|}\right\rangle,  \qquad I(x_k) = I\left(\frac{x_k}{\|x_k\|}\right).
$$
Together with \eqref{eq:344656456} these observations yield
$$
y_k  \in  \bigcap_{i \in I(\frac{x_k}{\|x_k\|})} \Argmax_{v\in C_i}\left\langle v, \frac{x_k}{\|x_k\|}\right\rangle \subseteq 	\bigcup_{x\in \Sph \atop h(x)>0} \bigcap_{i\in I(x)} \Argmax_{v\in C_i}\langle v, x\rangle,
$$
and hence the limit of the sequence $\{y_k\}$ must indeed belong to the left hand side of \eqref{eq:3.2mainspecial}.
\end{proof}

We are now ready to obtain a generalisation of Theorem~3.1 in \cite{OuterLimits}.

\begin{theorem}\label{thm:gen3.2lin} Let $f:X \to \R$ be as in \eqref{eq:min-function}, and in addition assume that for every $i\in I$ the function $f_i$ is piecewise affine, i.e.
$$
f_i (x) = \max_{j\in J_i} (\langle a_{ij}, x\rangle + b_{ij}) \quad \forall i \in I,
$$
where $J_i$'s are finite index sets for each $i\in I$. 	
Then
\begin{equation}\label{eq:affine}
	\bigcup_{p\in \Sph \atop f'(\bar x; p)>0} \bigcap_{i\in I(x,p)} \Argmax_{v\in \partial f_i(\bar x)}\langle v, p\rangle = \Limsup_{x\to \bar x\atop f(x) >f(\bar x)} \partial f(x).
\end{equation}
%	where, as before,
%	$$
%	I(x,p) = \left\{i_0\in I(x) \, |\, \max_{v\in \partial f_{i_0}(x)}\langle v,p\rangle = \min_{i\in I(x)}\max_{v\in \partial f_i(x)}\langle v,p\rangle\right\}.
%	$$	
\end{theorem}
\begin{proof} Observe that when the sets $C_i$ in Lemma~\ref{lem:techPH} are polyhedral, there is no need for the closure operation in \eqref{eq:3.2mainspecial}, since there are finitely many different faces of each subdifferential, and we therefore have a finite union of closed convex sets which is always closed.

To finish the proof it remains to note that convex polyhedral functions are locally positively homogeneous and coincide with a translation of their first order approximations in a sufficiently small neighbourhood of each point. Since outer limits of subdifferentials are local notions, it is clear that the application of Lemma~\ref{lem:techPH} to the directional derivatives of the active functions yields the required result.
\end{proof}

The next example is a practical demonstration of the construction given in  Theorem~\ref{thm:gen3.2lin}.

\begin{example}\label{eg:min-max} Let $f = \min\{f_1,f_2\}$, where $f_1,f_2:\R^2\to \R$ are piecewise linear max-functions,
$$
f_1(x,y) = \max \{2 x + y,-x + y,-x - y,-y\}, \quad
f_2(x,y) = \max \{3 x + y,y,2x - y,3x-y\}.
$$
The graphs of $f_1$, $f_2$ and $f$ are shown in Fig~\ref{fig:min-max-example-plots}.
\begin{figure}[ht]
{\centering 
\includegraphics[scale = 0.3]{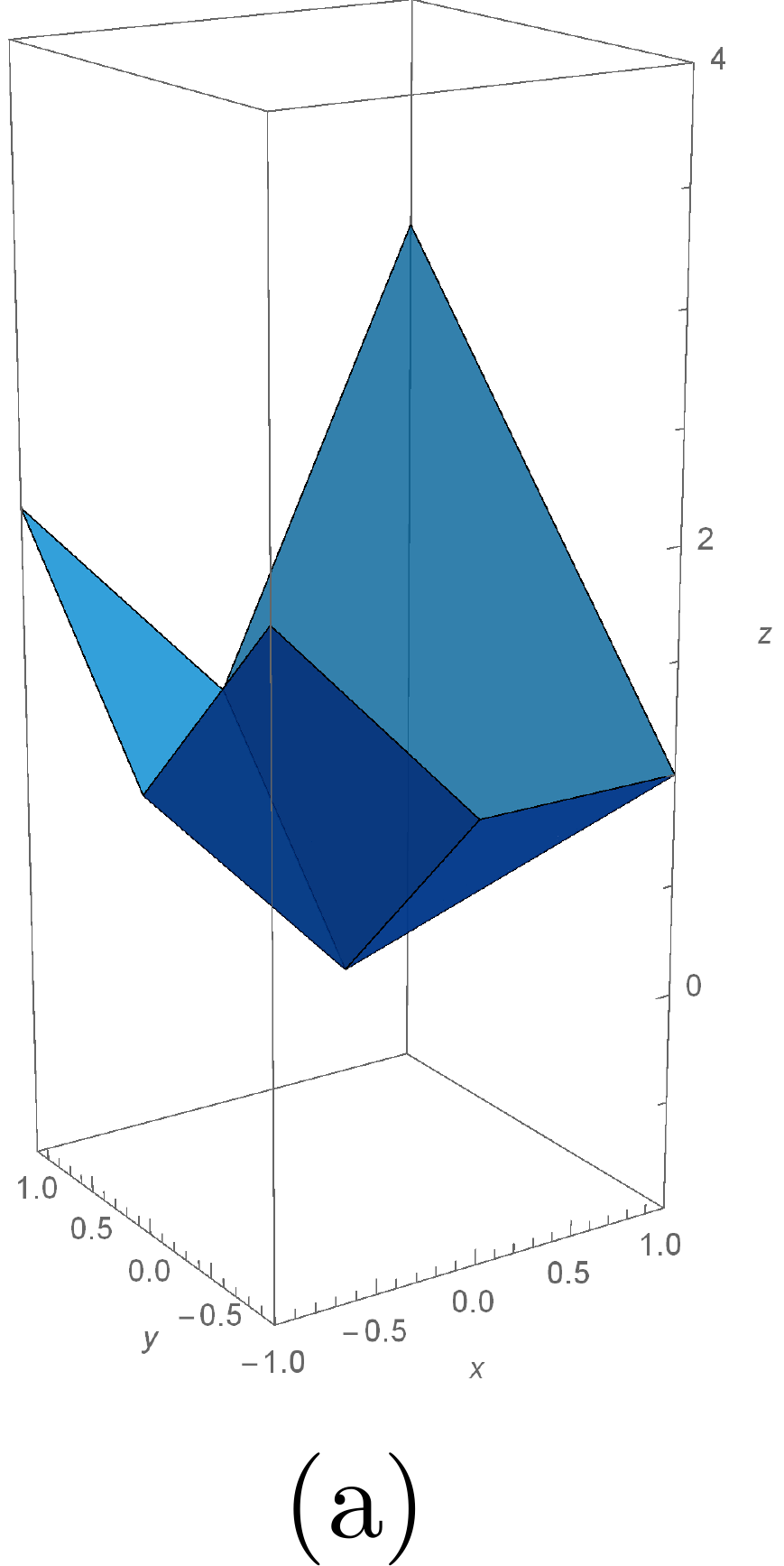} \quad  \includegraphics[scale = 0.3]{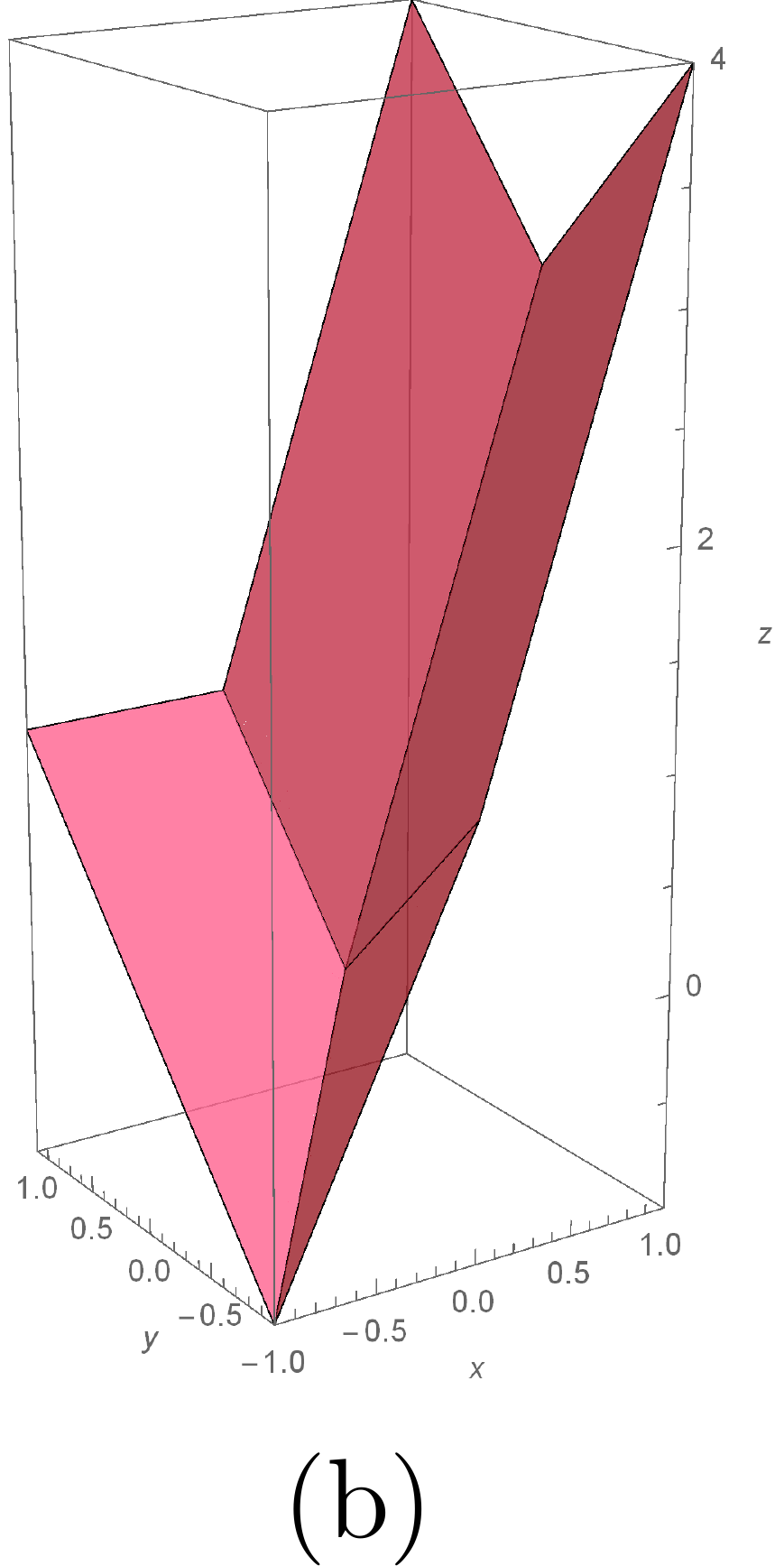} \quad  \includegraphics[scale = 0.3]{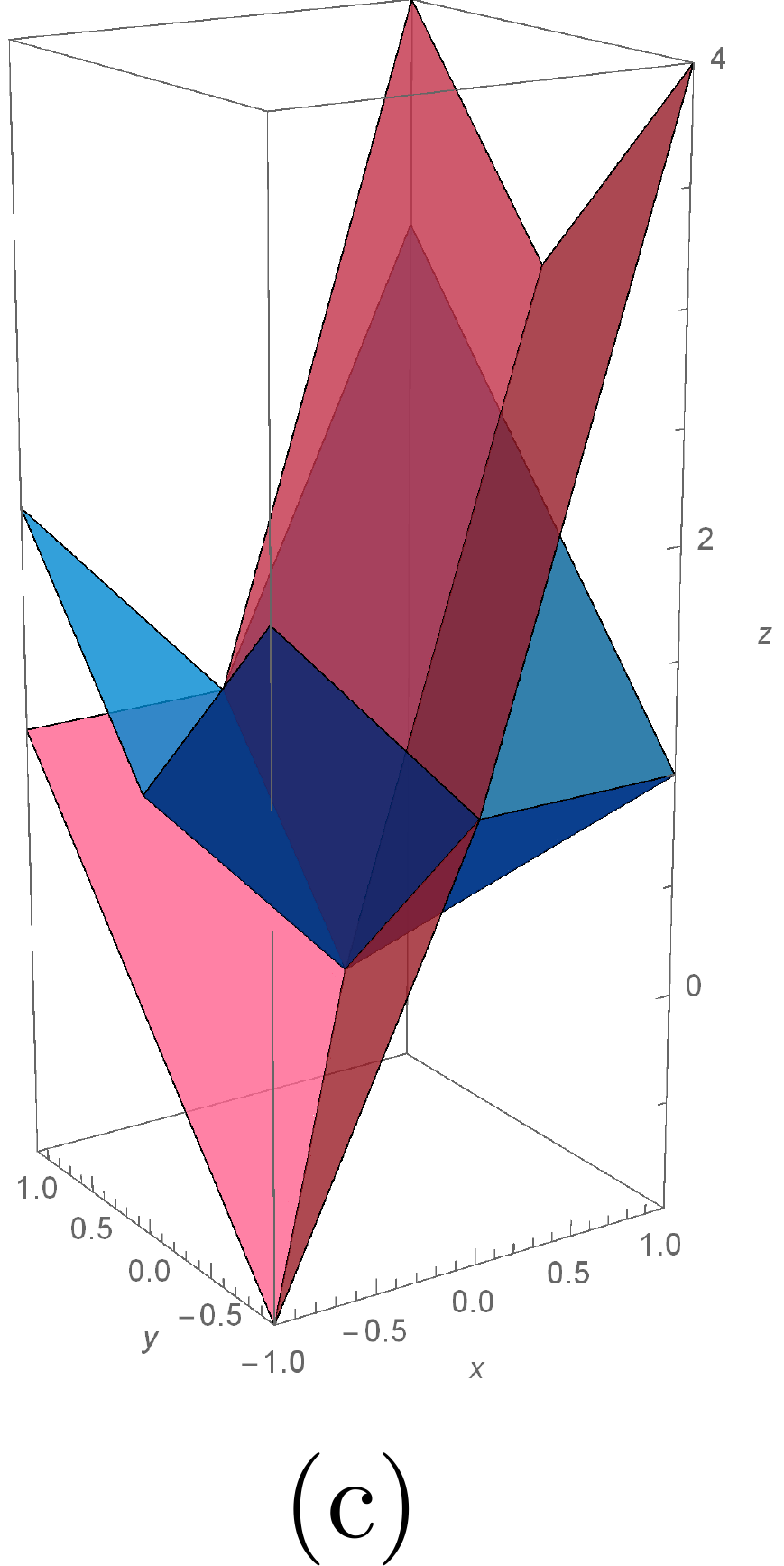} \quad  \includegraphics[scale = 0.3]{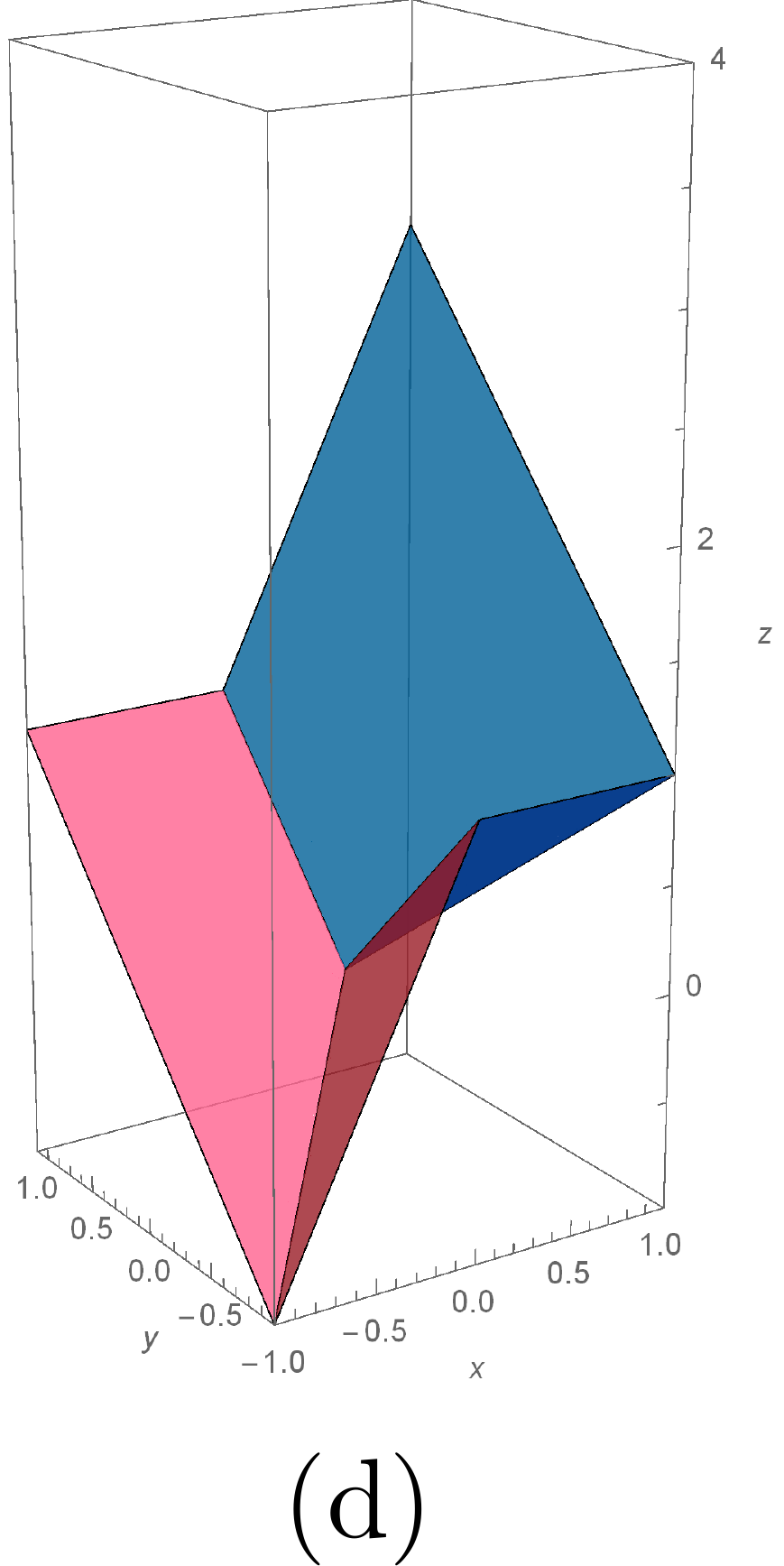} \quad  \includegraphics[scale = 0.3]{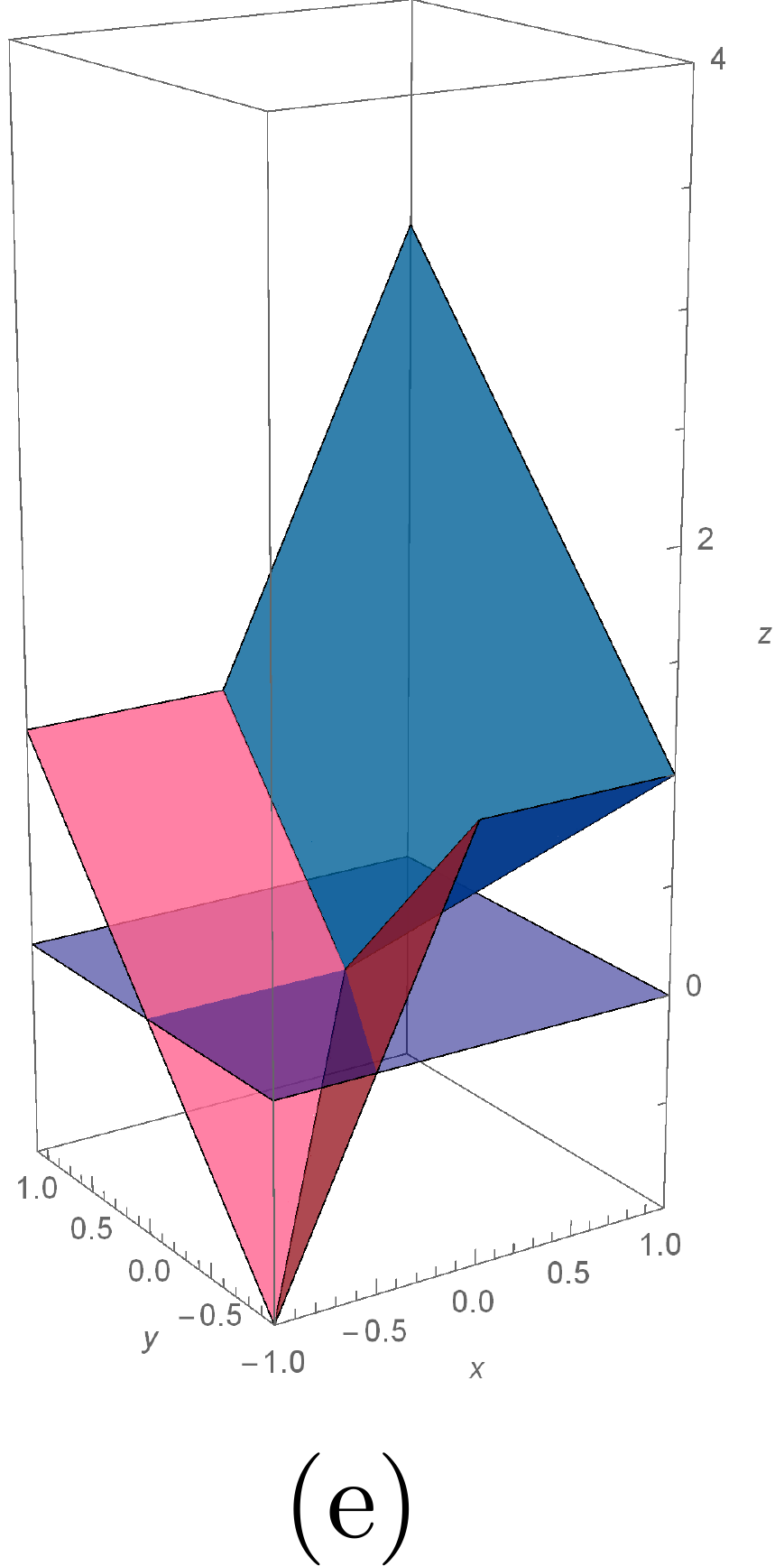}  \\}
\caption{Plots of the functions from Example~\ref{eg:min-max}, (a): $f_1$; (b): $f_2$; (c): $f_1$ and $f_2$ combined; (d): $f$; (e): $f$ with the plane $z =0$ demonstrating the relevant sublevel set.}
\label{fig:min-max-example-plots}
\end{figure}
The subdifferentials of $f_1$ and $f_2$ at zero are easy to compute:
\begin{align*}
\partial f_1(0) & = \co \{(2,1)^T, (-1,1)^T, (-1,-1)^T, (0,-1)^T\},\\
\partial f_2(0) & = \co \{(3,1)^T, (0,1)^T, (2,-1)^T, (3,-1)^T\}.
\end{align*}
The subdifferentials are shown in Fig.~\ref{fig:min-max-subdifferentials},
\begin{figure}[ht]
{\centering 
\includegraphics[scale = 1.1]{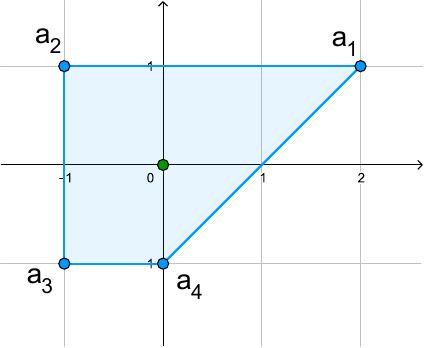} \;
\includegraphics[scale = 1.1]{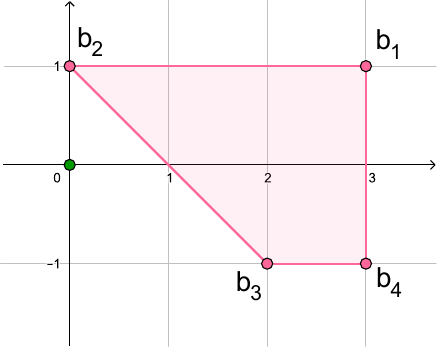} \;
\includegraphics[scale = 1.1]{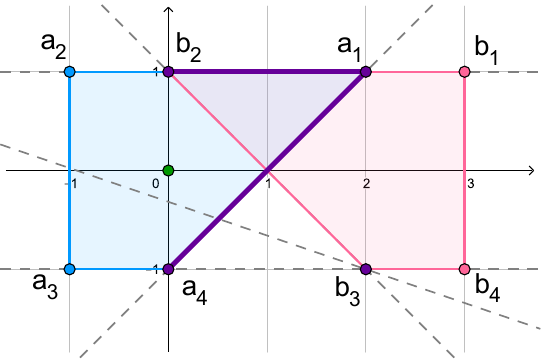} \\}
\caption{Example~\ref{eg:min-max}, from left to right: subdifferentials of the functions $f_1$ and $f_2$, and the construction of the outer limits of subdifferentials of $f$ at zero.}
\label{fig:min-max-subdifferentials}
\end{figure}
where $\partial f_1(0)= \co \{a_1,a_2,a_3,a_4\}$ and  $\partial f_2(0)= \co \{b_1,b_2,b_3,b_4\}$.

To carry out the calculation of the left-hand side expression in \eqref{eq:affine} it is enough to consider every vertex and face of the subdifferential and to verify that there is an exposing `minimal' hyperplane (line in our case) such that zero belongs to the relevant strictly negative subspace defined by this hyperplane. The relevant vector $p$ is the normal to such hyperplane. The resulting outer limit of subdifferentials is the set 
$$
\co \{a_1,b_2\} \cup \co \{a_1,a_4\} \cup \{b_3\}.
$$
We leave finer details to the reader. It is also not difficult to see from the graph that the two line segments correspond to the `convex' part of the graph, and the standalone point is the gradient of the linear part at the front of the plot in Fig.~\ref{fig:min-max-example-plots} (e) which is connected to the rest of the plot in a `concave' fashion.
\end{example}

We highlight here that the index-based representation that is valid for the convex case and that was used in Example~\ref{eg:marco-illustr} can not be generalised directly to the case of a min-max type function. This happens because the intersection of two polytopes of dimension 2 or higher can not always be represented via the convex hull of a subset of vertices of these polytopes, and hence the intersection in the left-hand side of \eqref{eq:affine} may not be representable as a convex hull of a selection of $a_{ij}$'s.

%The diagram in Fig.~\ref{Fig:minmax1}  gives a geometric demonstration on constructing the union of the minimal support faces as in Theorem~\ref{thm:gen3.2lin}.
%\begin{figure}[ht]
%{\centering \includegraphics[scale = 1]{pics/subdiff_indexSet_intersection.PDF}\\}
%\caption{The geometric construction of the active faces for the min-max function}
%\label{Fig:minmax1}
%\end{figure}

Note that in the case of a min-max type function the piecewise affine assumption is essential for the equality in \eqref{eq:affine} to hold. Consider the following semialgebraic example (unfortunately we could not recollect where the idea of this example came from).  
\begin{example}\label{eg:paraboloids}
Let $f= \min \{f_1, f_2\}$, where
$$
f_1(x,y)  = 1 - \frac{((x - 2)^2 + y^2)}{4}, \qquad f_2(x,y) =  -1 + ((x - 1)^2 + y^2).
$$
At the point $(x,y) = 0_2 = (0,0)$ we have
$$
f'(0_2, l) = \min \{f_1'(0_2, l), f_2'(0_2, l)\} = \min \{\langle \nabla f_1(0_2), l\rangle, \langle \nabla f_2(0_2), l\rangle\} =\min \{l_x, -2l_x\}\leq 0 \quad \forall l,
$$
therefore,
$$
\{l\in \R^2 \,|\, f'(0_2;l) >0\} = \emptyset,
$$
and so the expression on the left hand side of \eqref{eq:3.2main} produces the empty set.
\begin{figure}[ht]
\centering
\includegraphics[height=200pt]{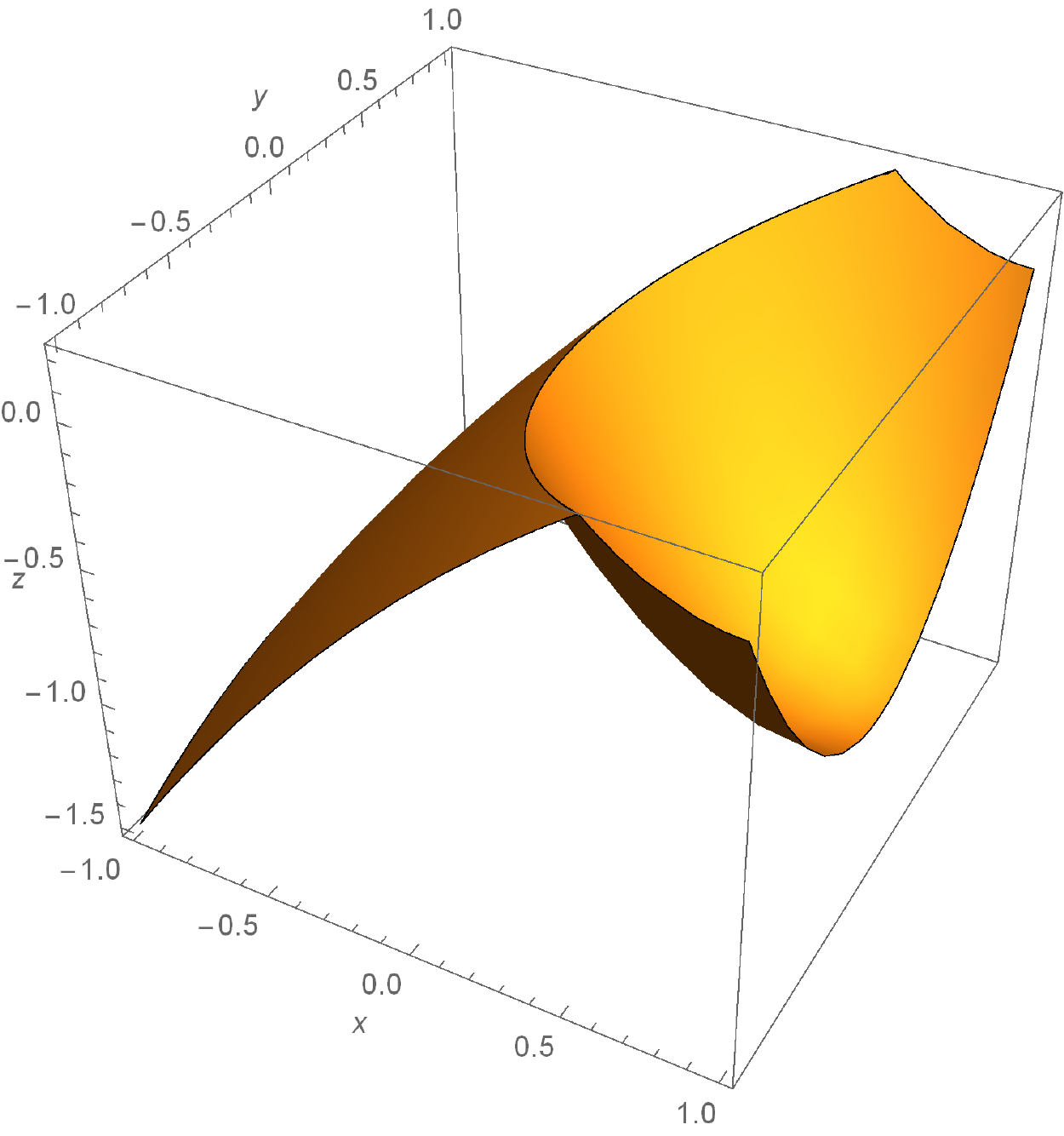}
\qquad
\includegraphics[height=200pt]{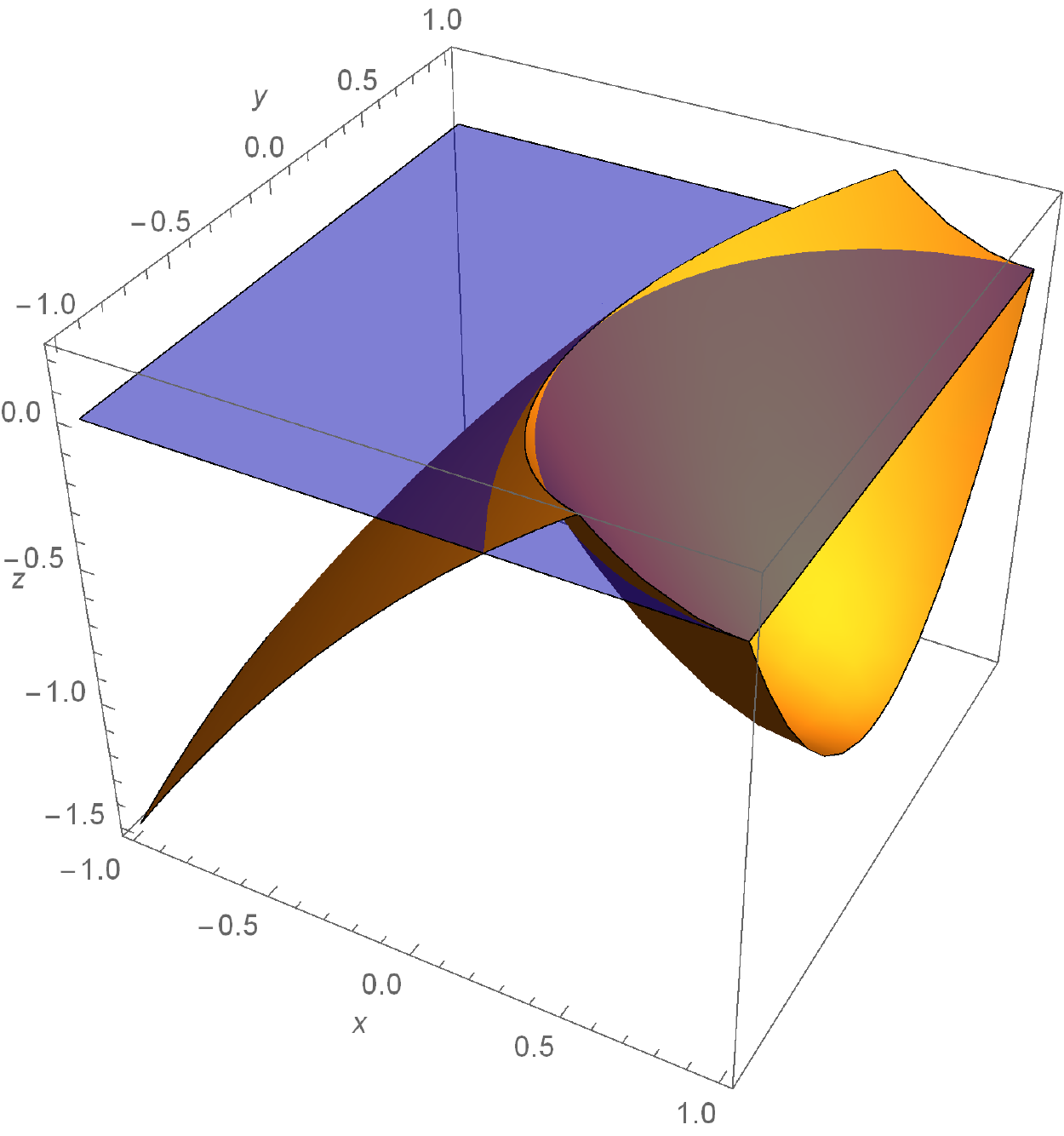}\\
\caption{Example}
\label{fig:example-affine-function}
\end{figure}

We now compute the outer limit directly. We have for the Fr\'echet subdifferential
$$
\partial f(x) =
\begin{cases}
\nabla f_1(x), &  f_1(x) <f_2(x),\\
\nabla f_2(x), &  f_1(x) >f_2(x),\\
\emptyset, &  f_1(x) =f_2(x), \nabla f_1(x) \neq \nabla f_2(x),\\
\nabla f_1(x), &  f_1(x) =f_2(x), \nabla f_1(x) = \nabla f_2(x).
\end{cases}
$$
Substituting the values and the gradients we have
$$
\partial f(x) =
\begin{cases}
(1-\frac{x}{2},\frac{y}{2})^T, & 12 x< 5 (x^2+ y^2)\\
(2x-2,2y)^T, &  12 x > 5 (x^2+ y^2)\\
\emptyset, &  12 x = 5 (x^2+ y^2).
\end{cases}
$$
It is not difficult to observe that
$$
\Limsup_{(x,y) \to 0_2\atop f(x,y) >f(0_2)} \partial f(x,y) \subseteq \Limsup_{(x,y) \to 0_2} \partial f(x,y) = \{(1,0)^T, (-2, 0)^T\}.
$$
On the other hand, observe that we can construct sequences of points converging to zero along the curves in the regions that correspond to $f_2>f_1>0$ and $0<f_2<f_1$ respectively. This works for points on the two curves
$$
3 x - x^2 - y^2 = 0 \qquad 11 x - 5 x^2 - 5 y^2 = 0
$$
shown in dashed and dotted lines in the last plot of Fig.~\ref{fig:example-affine}.
\begin{figure}[ht]
\centering
\includegraphics[height=200pt]{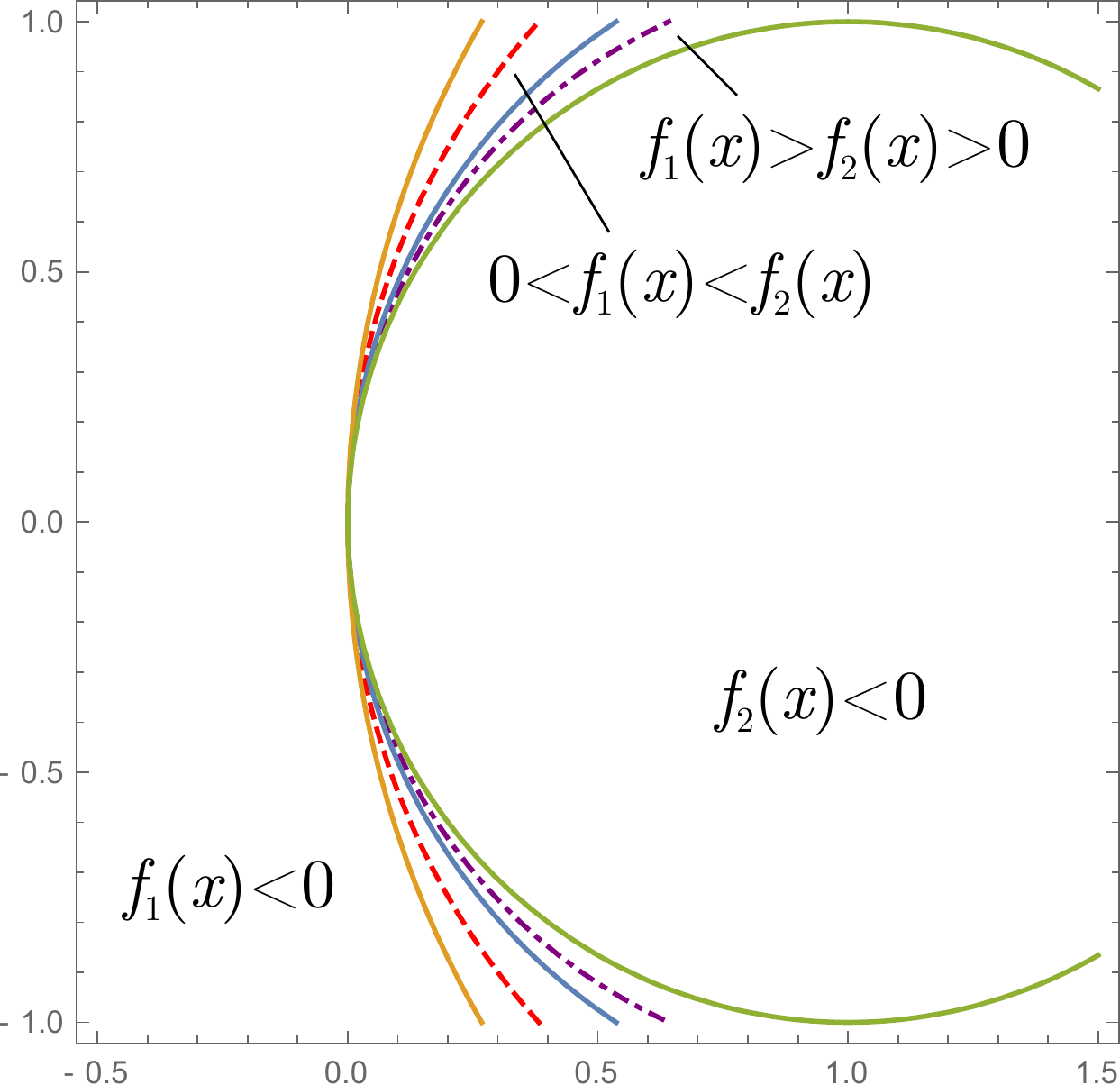}\\
\caption{Example~\ref{eg:paraboloids}.}
\label{fig:example-affine}
\end{figure}
\end{example}

\section{Acknowledgements}

The authors are grateful to Prof. Alex Kruger for insightful comments and corrections and also to Drs Kaiwen Meng  and Minghua Li for their comments and explanations that significantly improved both our understanding of the subject and the clarity of the exposition. We are also immensely grateful to the three anonymous Referees for their insightful comments, corrections and suggestions that significantly improved the quality of this paper. The research was partially supported by the Australian Research Council grant DE150100240.

\bibliographystyle{plain}
\bibliography{refs}

\end{document}